\documentclass[10pt]{amsart}

\pdfoutput=1
\usepackage{geometry} 
\usepackage{graphicx,url}
\usepackage{amssymb,amsmath, amsthm}
\usepackage{epstopdf}
\usepackage{subcaption}
\usepackage[utf8]{inputenc}

\usepackage{graphicx,url,hyperref}
\newtheorem{theorem}{Theorem}
\newtheorem{proposition}{Proposition}

\newtheorem{corollary}[theorem]{Corollary}

\title{Telegraph random evolutions on a circle }
\numberwithin{equation}{section}

\newcommand{\es}{{\mathds  E}}

\newcommand{\sgn}{\text{sgn}}
\newcommand{\de}{\mathrm{d}}

\usepackage{dsfont}
\usepackage{enumerate}
\theoremstyle{remark}
\newtheorem{remark}{Remark}[section]

\usepackage{xcolor}
\usepackage[percent]{overpic}

\author{Alessandro De Gregorio, Francesco Iafrate}
\address{Department of Statistical Sciences, ``Sapienza" University of Rome,
	P.le Aldo Moro, 5 - 00185, Rome, Italy}
\email{alessandro.degregorio@uniroma1.it}
\email{francesco.iafrate@uniroma1.it}

\allowdisplaybreaks
\date{\today}
\begin{document}
	
	\maketitle
	
	\begin{abstract}
		We consider the random evolution described by the motion of a particle moving on a circle alternating the angular velocities
		$ \pm c $ and
		changing rotation at Poisson random times, resulting in a telegraph process over the circle.
		We study the analytic properties of the semigroup it generates as well as its probability distribution. The asymptotic behavior of the wrapped process is also studied in terms of circular Brownian motion. 
		Besides, it is possible to derive a stochastic model for harmonic oscillators with random changes in direction and we give a diffusive approximation of this process. Furthermore, we introduce  some extensions of the circular telegraph model in  the asymmetric case
		and for non-Markovian waiting times as well. In this last case, we also provide some asymptotic considerations. 
		\end{abstract}

	\section{Introduction}
	
	Random evolutions are models for dynamical systems characterized by some degree of randomness
	stemming from the fact that the phenomenon takes place in a  ``random environment'' 
	which changes the way the system evolves over time. 
	The prototypical example of a random evolution is the \emph{telegraph process}, describing the motion on the line of a particle moving at constant speed $c$ but randomly flipping direction, forward and backward, at exponential times due for example to random collisions, which represent in this case the aforementioned effect of a random environment. This random process is very intuitive and it leads to describe well some stochastic motions arising in statistical physics and biology. For this reason	over the years many researchers studied the telegraph random model and proposed several generalizations (also involving jumps); see, e.g., \cite{ors1}, \cite{foong2}, \cite{ors2}, \cite{rat1}, \cite{beghin}, \cite{iac}, \cite{dicresc2}, \cite{dicresc3}, \cite{dicresc4}, \cite{dicresc6}, \cite{bog}, \cite{stadje}, \cite{zacks}, \cite{ors3}, \cite{dg},  \cite{lopez} and \cite{garra}. In  \cite{wat} and \cite{font} is deepened the asymptotic analysis of the telegraph motion. Some financial applications of the telegraph process are contained in \cite{kolrat}. In Section 2 we recall the main results about the telegraph process.

	Formally, such random systems can be modeled by means of a collection of semigroup operators $ \{\mathfrak T_v(t), t \geq 0\}, v \in S,$ acting on functions $ f $ in some Banach space, depending on a parameter $ v $ varying in a space $ S $ which describes the possible modes 
	of evolution of the system  (see, e.g., \cite{griegoh} and \cite{pinsky}) . In this abstract setting the parameter is then allowed to switch according to a random process, with values in $ S $, and the random evolution will be the composition of the corresponding sequence of semigroups. For example in the case of a telegraph process $ \{X(t), t\geq 0\} $ the underlying semigroup $ \mathfrak T_v(t) $ 
	is a translation with velocity $ v $, i.e. $ \mathfrak T_v(t) f(x) = f(x +vt) $ and the velocities change as a two valued continuous-time Markov process $ \{ V(t), t \geq 0\} $ taking values in $ S = \{-c,c\}$. 
	
	Our aim in this paper is to study the properties of a telegraph random evolution on a circle with radius $r$ denoted by $\mathbb S_r:=\{z\in\mathbb C: |z|=r\},0<r\leq 1$; that is the random process resulting from a uniform motion of a particle on a circle, undergoing a Markov switching of the direction of rotation. We assume that the starting point of the particle is $z\in\mathbb S_r$ which performs its displacements with angular speed $c.$ In the framework of random evolutions this process can be seen as an alternating application of a rotation semigroup $ \mathfrak T_v(t) $ with angle $ vt $ and angular velocities $ v = \pm c $; i.e. it can be represented as  $\mathfrak T_v(t) f(z) = f(z e^{ivt})$.
	The resulting circular random process $Z:= \{ Z(t),  t \geq 0\} $ can also be thought as a telegraph process  wrapped onto a circle by means of a complex exponential transformation
	\[
	Z(t) = ze^{i X(t)} \,, \quad t \geq 0.
	\]
	An in-depth examination of the \emph{circular telegraph process} from a functional-analytic point of view
	is carried out in Section 3, where we study the properties of the family of operators
	\[
	T_t f(z,v) = \es_{z,v}[f(Z(t),V(t))]  \,  \quad t \geq 0,
	\]
	where $z =re^{i\theta}, \, \theta\in[0,2\pi), 0<r\leq 1,\, v\in\{-c,c\}$, including the expressions of the generator and the resolvent. For $ z \mapsto f(z,v) $ analytic in the unit disc we are able to give explicit formulas for these operators. 
	As recalled in Section 2, the distribution of the telegraph process is related to a hyperbolic partial differential equation (the \emph{telegraph equation}) of the form
	\begin{equation*}\label{eq:telpde-intro}
	\frac{\partial^2 u}{\partial t^2}(x,t)+2\lambda \frac{\partial u}{\partial t}(x,t)=c^2 \frac{\partial^2 u}{\partial x^2}(x,t),\quad x\in \mathbb R, t\geq 0.
	\end{equation*}
	We show that an analogous connection can be established between the circular process $ Z $
	and a complex version of the telegraph equation on the unit disc
	\begin{equation*}\label{eq:cte-intro}
	\frac{\partial^2 p}{\partial t^2}+2\lambda \frac{\partial p}{\partial t}=-c^2 z^2[p''+p']=c^2\frac{\partial^2 p}{\partial\theta^2},\quad z=re^{i\theta}, z\neq 0, \theta\in[0,2\pi), 0<r<1.
	\end{equation*}
	Problems related to the \emph{complexification} of wave and telegraph equations have been studied in \cite{GGG} and \cite{GGG2}.
	
	In Section 4 we turn our attention to the probabilistic features of the process $ Z$. By exploiting some concepts coming from the literature concerning circular random variables (see, e.g., \cite{mardia}) we give the distribution, the joint moments and the characteristic function of the circular telegraph process.
	
	It is in general of great interest in the context of random evolution the study of their asymptotic behavior. In the case of the telegraph process it is well-known that it converges to a Brownian motion when the speed $ c $ and the rate of the exponential waiting times $ \lambda $ tend to infinity in such a way that $ \lambda/ c^2 \to 1 $. Analogous results hold in the general setting of random flights in $ \mathbb R^d $, see e.g. \cite{hor}, \cite{ghosh}. We analyze this type of weak convergence problems in Section 5. In particular it holds that the suitably scaled circular telegraph process  converges to a a circular Brownian motion
	$\mathfrak B(t):=e^{i B(t)},$ where $B:=\{B(t),t\geq 0\}$ is a standard Brownian motion.

	Starting from a circular telegraph process it is possible to build models for \emph{random harmonic oscillators}. Models of random harmonic oscillators are of great interest in physics and widely studied in the context of Langevin SDEs with a harmonic potential. They can be used to model for example the behavior of an oscillator in a fluid, possibly with a random damping term, See e.g. \cite{gitterman1} for a survey of the results, where the author also considers the case where the driving noise is a colored dichotomous noise, in particular a telegraph signal.
	In Section 6 we deal with an alternative model for random harmonic oscillators, based on the harmonic components of the circular telegraph process. These can be seen as harmonic oscillators changing direction at exponential times. Furthermore, we are able to explicitly derive the distribution function and the density of such processes as well as their diffusion approximation.
	
	Finally we explore different generalizations of the wrapped up telegraph process. A first possibility is to start from an asymmetric telegraph process: this allows  to have two different speeds along the two directions, possibly with different switching rates. 
	A second idea is to assume that the waiting time $ D $ distribution is heavy tailed instead of exponential, i.e.  $\mathds P(D>x) \sim x^{-\alpha}$,  $ 0 < \alpha < 2, \alpha \neq 1 $; this means that the waiting times can have infinite expectation and variance. In the first case we can give the distribution of the resulting circular telegraph process, while in the second case we are interested in the asymptotic behavior. In particular for $ 0 < \alpha < 1, $ we have that the convergence of the rescaled telegraph-type circular process with heavy-tailed waiting times converges to a time-changed symmetric circular stable process.

	\section{The telegraph process in a nutshell}
	
	The classical telegraph process $X:=\{X(t),t\geq 0\}$ describes the motion of a particle starting from the origin and moving on the real line with constant speed $c>0.$ The particle reverses its direction when a Poisson event occurs; i.e. the particle  moves forward and backward for a random exponential time. Hence, this random motion is a uniform motion with randomly alternating velocities and was introduced in \cite{gold} and \cite{kac}. The position of the telegraph process at time $t\geq 0$ is equal to
	\begin{equation}\label{eq:tel}
	X(t):=\int_0^t V(s) \de s=V(0)\sum_{k=1}^{N(t)-1}(-1)^{k-1}(T_k-T_{k-1})+(t-T_{N(t)})V(0) (-1)^{N(t)},
	\end{equation}
	where $\{V(t),t\geq 0\}$ represents a Markov chain with state space $\{-c,c\}$ (i.e. the velocity-jump process) with $V(t):=V(0)(-1)^{N(t)}$ and $V(0)$ uniformly distributed in $\{-c,c\}.$ Furthermore, $\{N(t),t\geq 0\}$ is a homogenous Poisson process with rate $\lambda >0$ and $\{T_k,k\geq 0\}$ is the sequence of Poisson times $(T_0:=0)$. The couple $\{(X(t),V(t)), t\geq 0\}$ is a Markov process while $X$ alone is not markovian. Let $\hat f:=(f(x,c),f(x,-c))'$ be a pair locally bounded, measurable functions belonging to  $C^1.$
 The mean value $\mathds E_{x,v}[f(X(t),V(t)]$ is a strongly continuous semigroup with infinitesimal generator given by (see, e.g., \cite{wat} and \cite{pinsky})
	$$\mathfrak G \hat f=(\mathfrak A +Q) \hat f,$$
	where $\mathfrak A:= \left(  \begin{matrix} 
	      c\frac{\partial}{\partial x} & 0 \\
	      0 & - c\frac{\partial}{\partial x}  \\
	   \end{matrix}\right)$
	   and 
	   $Q:= \left(  \begin{matrix} 
	      -\lambda & \lambda \\
	      \lambda & - \lambda  \\
	   \end{matrix}\right)$
	   represents the scattering component of $\mathfrak G$ in terms of matrix infinitesimal generator of $\{V(t),t\geq 0\}$.

Another interesting feature of the telegraph process is the connection with the hyperbolic partial differential equations. It is well-known that unique classical solution of the following Cauchy problem, involving the telegraph equation, 
		
	\begin{equation}\label{eq:telpde}
	\begin{cases}
	\frac{\partial^2 u}{\partial t^2}(x,t)+2\lambda \frac{\partial u}{\partial t}(x,t)=c^2 \frac{\partial^2 u}{\partial x^2}(x,t),\quad x\in \mathbb R, t\geq 0,\\
	u(x,0)=f(x), \quad \frac{\partial u}{\partial t}(x,0)=0,\quad f\in C^2(\mathbb R),
	\end{cases}
	\end{equation}
	is given by
	\begin{align}\label{eq:solteleq}
	u(x,t)&=\mathds{E}f(x+X(t))\\
	&=\frac12 e^{-\lambda t}[f(x+ct)+f(x-ct)]+\int_{-ct}^{ct} \mu(y,t) f(x+y)\de y\notag
	\end{align}
	where
	\begin{align}\label{eq:acpart}
	\mu(x,t):=\frac{\lambda e^{-\lambda t}}{2} \left[\frac1c I_0\left(\frac\lambda c\sqrt{c^2t^2-x^2}\right)+t \frac{ I_1\left(\frac\lambda c\sqrt{c^2t^2-x^2}\right)}{\sqrt{c^2t^2-x^2}}\right],
	\end{align}
	with $I_\nu(x):=\sum_{k=0}^\infty \frac{(z/2)^{\nu+2k}}{k!\Gamma(k+\nu +1)},\nu \in\mathbb C,$ representing the so-called modified Bessel function. Furthermore, the probability law of $X(t),$ for any $t\geq 0,$ is given by
	\begin{align}\label{eq:lawtel}
	p(\de x,t):= \mathds{P}(X(t)\in \de x)=\frac12 e^{-\lambda t}[\delta_{- ct}(\de x)+\delta_{ct}(\de x)]+\mu(x,t) \mathds{1}_{|x|<ct}\,\de x
	\end{align}
	where $\mu(x,t)\mathds{1}_{|x|<ct}\, \de x$ represents the absolute continuous component of the probability distribution of the telegraph process and $\delta_x(B)$ is the Dirac measure. An abstract version of the above problem is studied in \cite{griegoh}. Furthermore, we note that $p(\de x,t)$ admits a singular part for any $t>0$ since $\mathds{P}(X(t)=ct)=\mathds{P}(X(t)=-ct)=e^{-\lambda t}$ (i.e. if $N(t)=0$) and then we have that $X(t)\in [-ct,ct] ,$  $\mathds{P}-$a.e..  We observe that $p(\de x,t)$ can be viewed as a probabilit measure involving a density function in the space of generalized functions $\mathfrak D'(\mathbb R)$ (i.e. the dual space of all infinte differentiable function on $\mathbb R$ with compact support). Let $\xi \in\mathbb R,$ the characteristic function of $X(t), t\geq 0,$ becomes
	\begin{align}\label{eq:cftp}
	\mathds E e^{i \xi X(t)}&=\int_{\mathbb R}e^{i x\xi} p(x,t)\de x\\
	&= e^{-\lambda t}\left[\cos(t\sqrt{\xi^2c^2-\lambda^2})+\frac{\lambda }{\sqrt{\xi^2c^2-\lambda^2}}\sin(t\sqrt{\xi^2c^2-\lambda^2})\right]\notag\\\
	&= e^{-\lambda t}\left\{\left[\cosh(t\sqrt{\lambda^2-\xi^2c^2})+\frac{\lambda }{\sqrt{\lambda^2-\xi^2c^2}}\sinh(t\sqrt{\lambda^2-\xi^2c^2})\right]\mathds{1}_{|\xi|\leq \lambda/c}\right.\notag\\
	&\quad\left.+\left[\cos(t\sqrt{\xi^2c^2-\lambda^2})+\frac{\lambda }{\sqrt{\xi^2c^2-\lambda^2}}\sin(t\sqrt{\xi^2c^2-\lambda^2})\right]\mathds{1}_{|\xi|>\lambda/c}\right\}\notag.
	\end{align}
	Furthermore, the conditional density functions are given by (see \cite{dgos})
	\begin{equation}\label{eq:disttp}
	\mu_n(x,t):=\frac{\mathds P( X(t)\in \de x|N(t)=n)}{\de x}=\begin{cases}
	\frac{\Gamma(n+1)}{(\Gamma(\frac{n+1}{2}))^22^{n}ct}\left(1-\frac{x^2}{c^2t^2}\right)_+^\frac{n-1}{2},& n\, \text{odd},\\
	\frac{\Gamma(n+1)}{\Gamma(\frac n2+1)\Gamma(\frac n2)2^{n}ct}\left(1-\frac{x^2}{c^2t^2}\right)_+^{\frac n2-1},&  n\, \text{even}.
	\end{cases}
	\end{equation}

We also provide the following result regarding the conditional probability law of $(X(t),V(t)), t\geq 0.$
	\begin{theorem} 
		For any $t\geq 0,$ we have that for $v\in\{-c,c\}$
		\begin{align}
		\mathds P(X(t)\in \de x, V(t)=v|V(0)=v)= e^{-\lambda t}\left[\delta_{v t}(\de x)+\frac{\lambda}{2c}(ct + x\text{sgn}(v)) \frac{I_1(\frac\lambda c\sqrt{c^2t^2-x^2})}{\sqrt{c^2t^2-x^2}}\mathds 1_{|x|<ct}\de x\right]
		\end{align}
		and
		\begin{align}
	\mathds P(X(t)\in \de x, V(t)= v|V(0)= -v)= e^{-\lambda t}\frac{\lambda}{2c} I_0\left(\frac\lambda c\sqrt{c^2t^2-x^2}\right)\mathds 1_{|x|<ct}\de x.
		\end{align}
	\end{theorem}
	\begin{proof}
		We recall the following result contained in \cite{dgos}
		\begin{equation}\label{eq: probcon1}
		\mathds P(X(t)\in \de x|N(t)=2k+2, V(0)=\pm c)=\de x\frac{(2k+2)!}{(2ct)^{2k+2}}\frac{(ct\pm x)^{k+1}(ct\mp x)^k}{k! (k+1)!}\mathds 1_{|x|<ct}.
		\end{equation}
		Therefore, by exploiting the result \eqref{eq: probcon1}, we can write down
		\begin{align*}
		&\mathds P(X(t)\in \de x, V(t)=\pm c|V(0)=\pm c)\\
		&=\mathds P(X(t)\in \de x, N(t)=0|V(0)=\pm c)+ \sum_{k=0}^\infty  \mathds P(X(t)\in \de x, N(t)=2k+2|V(0)=\pm c)\\
		&=e^{-\lambda t}\delta_{\mp ct}(\de x)+\sum_{k=0}^\infty \mathds P(N(t)=2k+2) \mathds P(X(t)\in \de x| N(t)=2k+2, V(0)=\pm c)\\
		&=e^{-\lambda t}\delta_{\mp ct}(\de x)+\sum_{k=0}^\infty e^{-\lambda t}\frac{(\lambda t)^{2k+2}}{(2k+2)!}\frac{(2k+2)!}{(2ct)^{2k+2}}\frac{(ct\pm x)^{k+1}(ct\mp x)^k}{k! (k+1)!}\mathds 1_{|x|<ct}\de x\\
		&=e^{-\lambda t}\delta_{\mp ct}(\de x)+e^{-\lambda t} \frac{\lambda}{2c}\frac{(ct \pm x)}{\sqrt{c^2t^2-x^2}}\sum_{k=0}^\infty \left( \frac{\lambda}{2c}\right)^{2k+1}\frac{1}{k!(k+1)!}(\sqrt{c^2t^2-x^2})^{2k+1}\mathds 1_{|x|<ct}\de x\\
		&=e^{-\lambda t}\left[\delta_{\mp ct}(\de x)+\frac{\lambda}{2c}(ct \pm x) \frac{I_1(\frac\lambda c\sqrt{c^2t^2-x^2})}{\sqrt{c^2t^2-x^2}}\mathds 1_{|x|<ct}\de x\right].
		\end{align*}
		Analogously, we resort the following result proved in \cite{dgos}
		\begin{equation}\label{eq: probcon2}
		\mathds P(X(t)\in \de x|N(t)=2k+1, V(0)=\mp c)=\de x\frac{(2k+1)!}{(2ct)^{2k+1}}\frac{(c^2t^2-x^2)^k}{(k!)^2}\mathds 1_{|x|<ct}.
		\end{equation}
		Therefore, by exploiting the result \eqref{eq: probcon2}, we get
		\begin{align*}
	\mathds P(X(t)\in \de x, V(t)=\pm c|V(0)=\mp c)
		&=\sum_{k=0}^\infty  \mathds P(X(t)\in \de x, N(t)=2k+1|V(0)=\mp c)\\
		&=\sum_{k=0}^\infty \mathds P(N(t)=2k+1)\mathds P(X(t)\in \de x| N(t)=2k+1, V(0)=\mp c)\\
		&=\sum_{k=0}^\infty e^{-\lambda t}\frac{(\lambda t)^{2k+1}}{(2k+1)!}\frac{(2k+1)!}{(2ct)^{2k+1}}\frac{(c^2t^2-x^2)^k}{(k!)^2}\mathds 1_{|x|<ct}\de x\\
		&=e^{-\lambda t} \frac{\lambda}{2c}\sum_{k=0}^\infty \left( \frac{\lambda}{2c}\right)^{2k}\frac{1}{(k!)^2}(\sqrt{c^2t^2-x^2})^{2k}\mathds 1_{|x|<ct}\de x\\
		&=e^{-\lambda t}\frac{\lambda}{2c} I_0\left(\frac\lambda c\sqrt{c^2t^2-x^2}\right)\mathds 1_{|x|<ct}\de x,
		\end{align*}
		which concludes the proof.
	\end{proof}

	\section{The wrapped up telegraph process}

	We introduce the wrapped up telegraph process on a circle; i.e. a circular motion with random angular velocity. Let us consider a particle performing random displacements on a circle $\mathbb S_r:=\{z\in \mathbb C:|z|=r\},  0<r\leq1,$ (i.e. arcs of circumference) with constant angular speed $c>0$ and starting from a point on $\mathbb S$ denoted by  $z:=re^{i\theta},\theta\in[0,2\pi).$  The particle initially chooses, with the same probability, to start its motion with a rotation $ze^{ict}$ or $ze^{-ict}$ till the first time $t$ where the particle reverses its motion. A homogeneous Poisson process with $\lambda>0,$ governs the changes of rotation; i.e. at each Poisson event the particle changes its rotation and moves in the opposite way till a new Poisson time.
	
	More precisely, let $Z(t)$ be the position of the particle on $\mathbb S_r$ at time $t\geq 0,$ with $Z(0)=z.$ Let $\{V(t),t\geq 0\}$ be the jump process introduced for the standard telegraph process in the previous section.  The particle initially moves choosing with probability 1/2 a rotation given by $ze^{iv}, v\in\{-c,c\}$; i.e. $V(0)=v.$  At time $0\leq t<T_1$ the random motion reaches the position 
	\begin{align*}
	&Z(t)=ze^{i v t},\\
	&V(t)=V(0)=v.
	\end{align*}
	In $T_1$ a Poisson event occurs and then the motion invert its rotation, that is $V(T_1)=-v,$ and performs a displacement on the arc of the circle always with constant angular speed $c$  till a new Poisson event; i.e. for $T_1\leq t<T_2$ the position of the particle on $\mathbb S_r $  is
	\begin{align*}&Z(t)=ze^{iv T_1} e^{-i v (t- T_1)},\\
	&V(t)=V(T_1)=-v.
	\end{align*}
	The particle continues to move following the same random scheme. We observe that in the circular telegraph process the translations $x\pm ct$ are replaced with the rotations $ze^{\pm ict}$.
	Therefore, bearing in mind \eqref{eq:tel}, we define $Z:=\{Z(t),t\geq 0\}$ as the telegraph process on $\mathbb S_r$ starting from $Z(0)=z$ where 
	\begin{equation}\label{eq:circtel}
	Z(t):=ze^{i X(t)}.
	\end{equation}
	Therefore the  standard telegraph process $\{X(t),t\geq 0\}$ appearing \eqref{eq:circtel} represents the random angle subtended to the arc descried by $Z$ while $\{V(t),t\geq 0\}$ denotes its  random angular velocity leading two modes of the particle evolution. 
	The imaginary exponential function $\mathbb R \ni x\mapsto e^{ix}$ is a (surjective) morphism of topological groups from the real line $\mathbb R$ to the unit circle $\mathbb{S}$. Therefore the sample paths of $\{Z(t),t\geq 0\}$ are obtained by wrapping up around $\mathbb S_r$ the trajectories of the standard telegraph process $\{X(t),t\geq 0\}$ by means of the imaginary exponential function. 
	A sample path of $ X(t) $ and the corresponding wrapped path $ Z(t) $ are shown in \autoref{fig:circle-path}.

	Equivalently, we can define \eqref{eq:circtel} as the solution of the following stochastic differential equation
	$$\frac{ \de Z(t)}{\de t}=\mathfrak V(t), \quad Z(0)=z, V(0)=v,$$
	with solution
	$$Z(t)=\mathfrak V(0)+\int_0^t \mathfrak V(s)\de s,$$
	where $\{\mathfrak V(t),t\geq 0\},$  $\mathfrak V(t):= i V(t) Z(t),$  denotes the velocity process of $Z.$ We observe that  $A(t):=\frac{\de \mathfrak V(t)(t)}{\de t}=\frac{\de^2 Z(t)}{\de t^2}=-V^2(t) Z(t)=-c^2 Z(t), t\geq 0,$ $\mathds P$-a.s., is the random centripetal acceleration of $Z$, and then
	$$\mathfrak V(t)=\mathfrak V(0)+\int_0^t A(s)\de s.$$

		The vectors $ Z(t) $ and $ \mathfrak V(t) $ corresponding to a piece of the sample path of the telegraph process are depicted in \autoref{fig:circle-vecs}.

	\begin{figure}[h]
		\centering
		\begin{overpic}[width=.5\linewidth, tics=10]{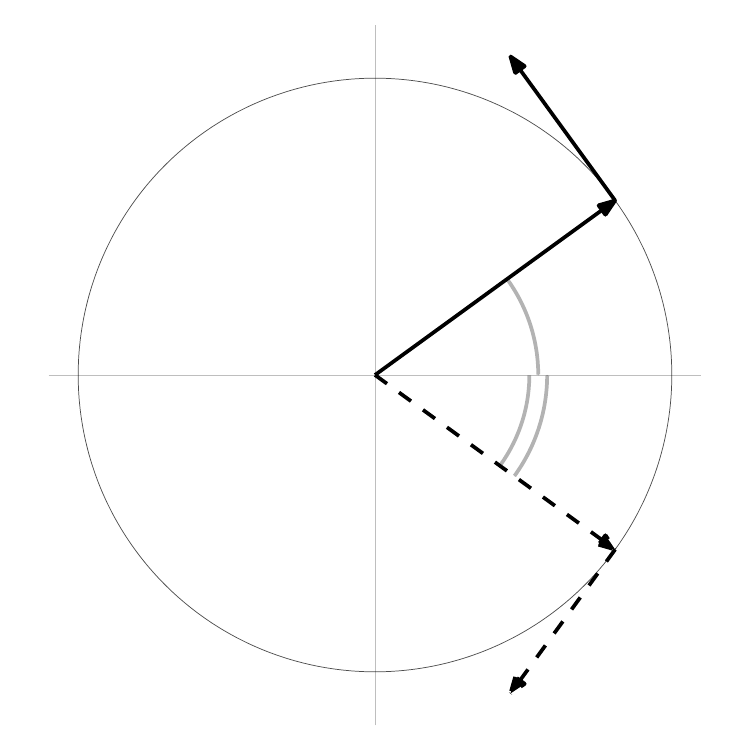}
			\put (50,70) {\small $Z(t)=e^{i  X(t)}$}
			\put (45,27) {\small $Z(t)=e^{-i X(t)}$}
			\put (72,92) {\small $\mathfrak V(t) = ic Z(t)$}
			\put (72,55) {\footnotesize $X(t)=ct$}
			\put (73,41) {\footnotesize $X(t)=-ct$}
			\put (73,6) {\small $\mathfrak V(t) =- ic Z(t)$}
		\end{overpic}	
		\caption{Plot of the vector components of the circular telegraph process for $ t < T_1 $ in the case $ V(0)=+c$ (solid line) or $ V(0)=+c$ (dashed line). Here the telegraph process $ X(t) $ represents the angle spanned by the trajectory.}
		\label{fig:circle-vecs}
	\end{figure}
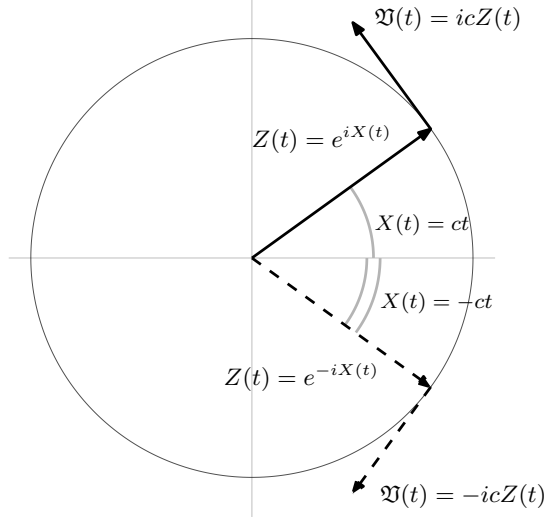

	Since we are describing a uniform circular random motion, as expected, the intensity of $\mathfrak V(t)$ and $A(t)$ are given by $|\mathfrak V(t)|=|V(t)|=c$ and $|A(t)|=c^2.$
	\begin{remark}
		Alternatively, we can introduce the circular telegraph process by means of polar coordinates in $\mathbb R^2;$ i.e.
		\begin{equation}
		Z(t):=  \left(  \begin{matrix}
		Z_1(t)\\
		Z_2(t)
		\end{matrix} \right):=\left(  \begin{matrix} 
		r\cos(\theta+X(t)) \\
		r\sin(\theta+X(t))  \\
		\end{matrix}\right),
		\end{equation}
		and
		\begin{equation}
		\mathfrak V(t):=\left(  \begin{matrix}
		\mathfrak V_1(t)\\
		\mathfrak V_2(t)
		\end{matrix} \right)= \left(  \begin{matrix} 
		-r V(t)\sin(\theta+X(t)) \\
		rV(t)\cos(\theta+X(t))  \\
		\end{matrix}\right),
		\end{equation}
		where $\theta$ is the initial angle. The random coordinates $Z_1(t)$ and $Z_1(t)$ turn out to be the harmonic motions related to $Z$  (in \autoref{fig:circle-path} are displayed two possible sample paths of these processes). This issue will be discussed in Section 6.
	\end{remark}

In what follows, we are going to present some results concerning the properties of the family of the semigroup operators associated to the wrapped up telegraph process. Let $(Z(0),V(0))=(z,v)$ where $z:=re^{i\theta}, \theta\in[0,2\pi), 0<r\leq 1,$ and $v\in\{-c,c\}.$ Therefore $\{(Z(t),V(t)), t\geq 0\}$ represents an homogenous strong Markov process (see, e.g., Theorem (25.5) in \cite{dav}) on $(\Omega, \mathcal F, (\mathcal F_t)_{t\geq 0})$ with state space $\mathbb{S}_r\times\{-c,c\}$. Let $\es_{z,v}[\cdot]$ the mean value with respect to the probability law of $\{(Z(t),V(t)), t\geq 0\}$ with starting point $(z,v)\in\mathbb G := \mathbb D  \times\{-c,c\}$, where 
and $\mathbb D=\{z\in\mathbb C:|z|<1\}$ denotes the unit disc of the complex plane. We assume that if $ z=0 $
then $ Z(t) = 0 $ for all $ t>0$.

	Let $\mathbb E=\mathbb R,\mathbb C;$  $B (\mathbb G,\mathbb E)$ stands for  the space of  $\mathbb E$-valued bounded, measurable functions whose domain is $\mathbb G$.
	Let $$A(\mathbb D):=\{f:\overline{\mathbb D}\to \mathbb C; f \,\text{analytic in}\, \mathbb D ,\, \text{continuous on}\, \overline{\mathbb D} \}$$ endowed with the uniform norm $||f||=\sup_{z\in\overline{\mathbb D}  }|f(z)|$. $(A(\mathbb D), ||\cdot||)$ represents a Banach space.

	\begin{figure}
		\centering
		\begin{subfigure}{.6\textwidth}
			\begin{overpic}[width=\linewidth,tics=10, 
				trim=0 100 0 100, clip]{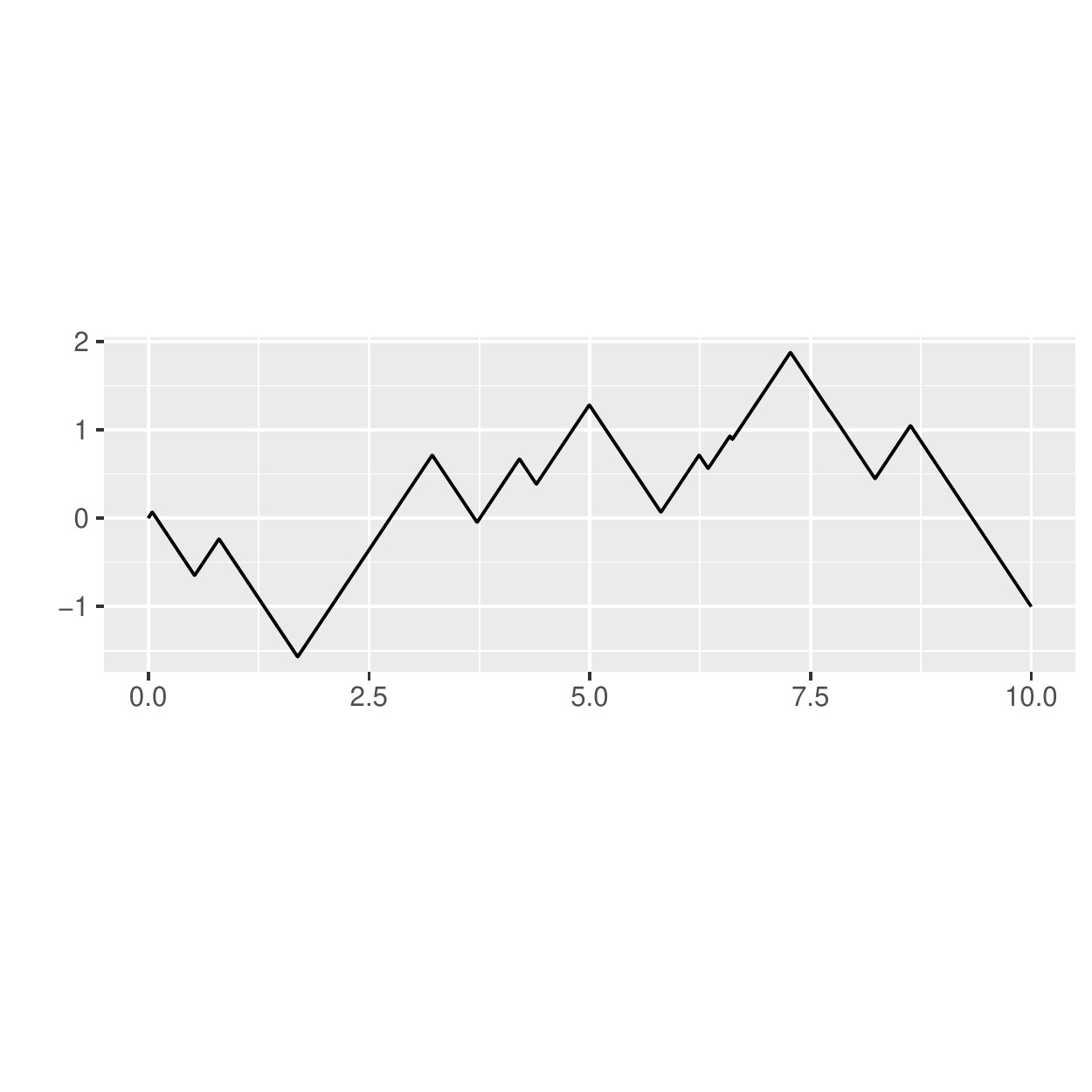}
				\put (53,3) {\footnotesize $t$}
				\put (-2,25) {\footnotesize $X(t)$}
			\end{overpic}
		\end{subfigure}

		\begin{subfigure}{.5\textwidth}
			\begin{overpic}[width=\linewidth,tics=10]{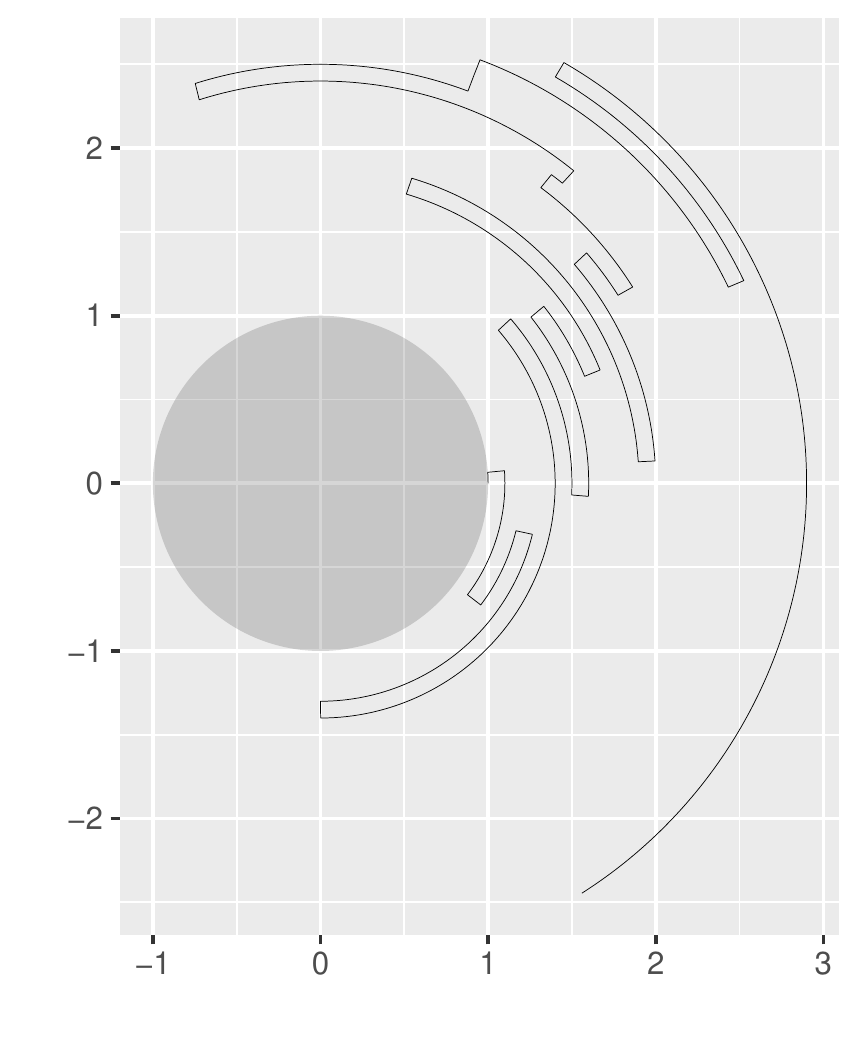}
				\put (29,53) {\footnotesize$\mathbb S_1$}
				\put (43,2) {\footnotesize$\text{Re}(z)$}
				\put (82,53) {\footnotesize$\text{Im}(z)$}
				\put (70,25) {\footnotesize$Z(t)$}
			\end{overpic}
		\end{subfigure}
		
			\begin{subfigure}{.5\textwidth}
			\begin{overpic}[width=\linewidth,tics=10]{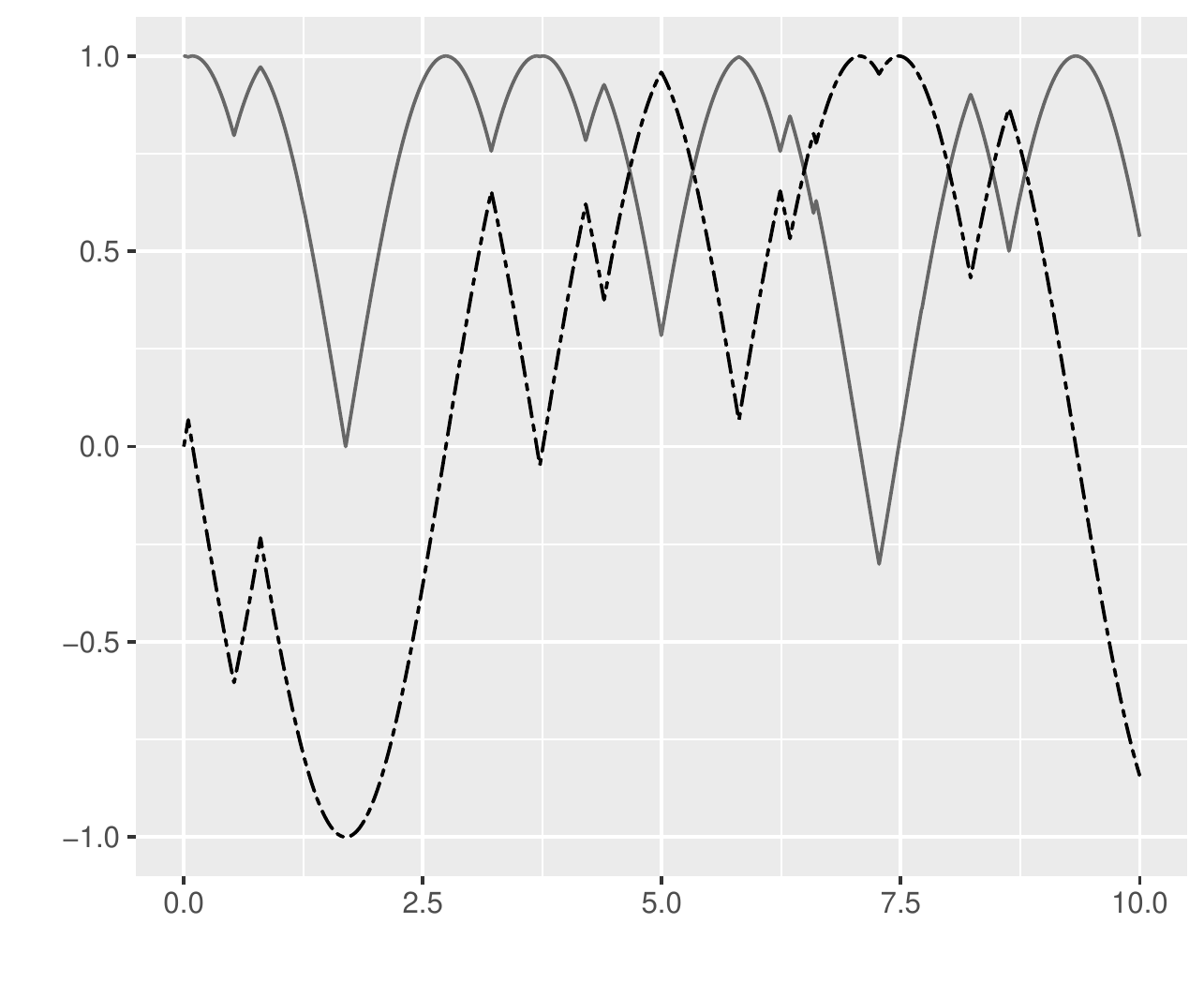}
				\put (54,1) {\footnotesize $t$}
				\put (14,63) {\footnotesize $Z_1(t)$}
				\put (14,17) {\footnotesize $Z_2(t)$}
			\end{overpic} 
		\end{subfigure}
		\caption{A sample path of the telegraph process $X(t)$ (top) and the corresponding wrapped path of $Z(t)$ (center), artificially displayed with increasing radii in order to visually distinguish between overlapping arc pieces. (Each vertical jump depicts a Poisson event, and not an actual jump of the sample path: formally, the plot of $(1+\epsilon N(t)) Z(t)$ is shown, where $ \epsilon >0 $ is the height of the jump). The plot on the bottom shows the sample paths of the associated harmonic components.}
		\label{fig:circle-path}
	\end{figure}
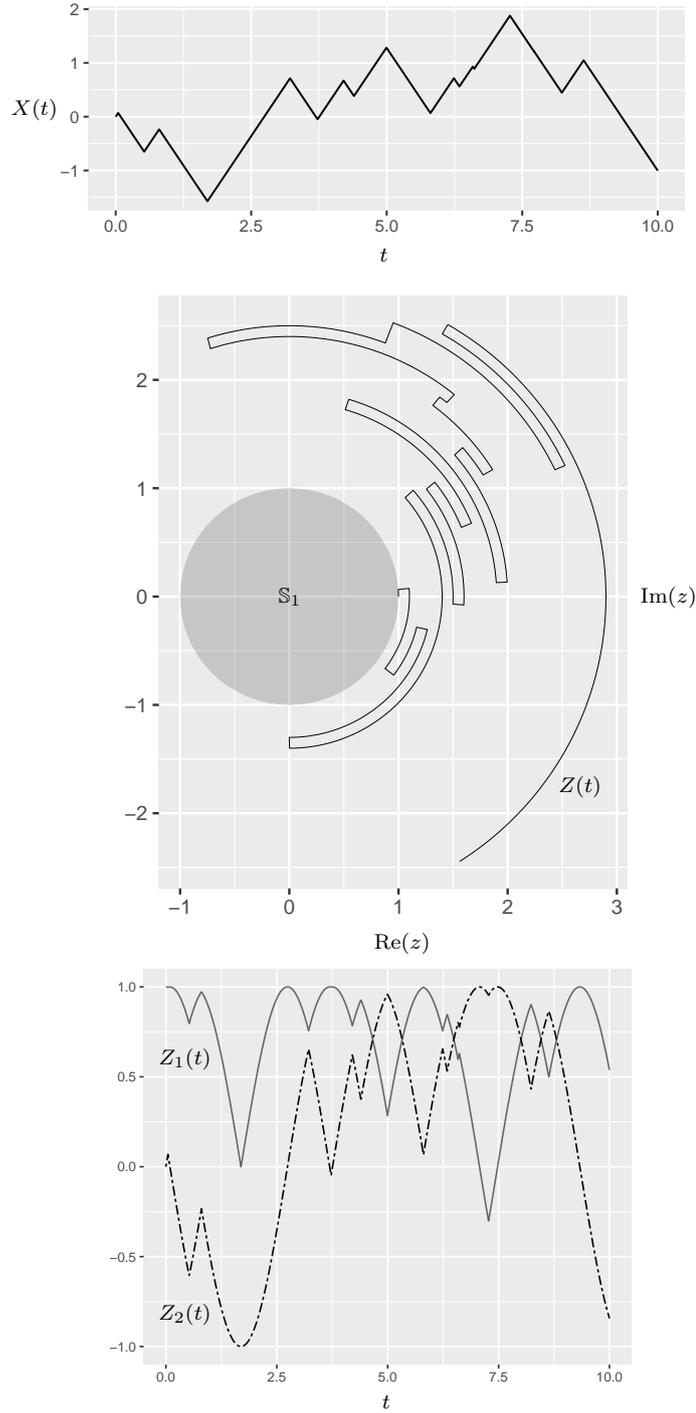
	
	\begin{theorem} Let $f\in B(\mathbb G, \mathbb C)$ and  the family $\{T_t f, t\geq 0\}$ of bounded linear operators defined as follows
		$$T_t f(z,v):=\es_{z,v}[f(Z(t),V(t))],\quad (z,v)\in \mathbb G. $$
		We have that:
		\begin{enumerate}[(i)]
			\item $T_t f\in B(\mathbb G,\mathbb C);$
			
			\item  if $z\mapsto f(z, v)\in A(\mathbb D)$ then $z\mapsto T_t f(z,v)\in A(\mathbb D)$ and 
			
			\begin{align}\label{eq:teoolomof}
			&T_t f(z,v)=e^{-\lambda t}\sum_{k=0}^\infty z^k[d_k(t,v)+e_k(t,v)], \quad z\in \mathbb D,
			\end{align}
			where
			$$d_k(t,v):=a_k(v)\left[e^{ikvt}+\cos(t\sqrt{c^2k^2-\lambda^2}) +i \text{\sgn}(v)\frac{ck \sin(t\sqrt{c^2k^2-\lambda^2})}{\sqrt{c^2k^2-\lambda^2}}-(1+\text{\sgn}(v))\cos(kct)\right]$$
			and
			$$e_k(t,v):=a_k(-v)\lambda\frac{ \sin(t\sqrt{c^2k^2-\lambda^2})}{\sqrt{c^2k^2-\lambda^2}}.$$
			
			\item  $u(z,v,t):= T_t f(z,v), v\in\{-c,c\},$ represent the solution of the following system of integral equations
			\begin{align}\label{eq:inteq}
			u(z,v,t)&=e^{-\lambda t} f(ze^{ivt},v)+\lambda \int_0^t  e^{-\lambda s} u(ze^{ivs},-v,t-s)\de s,\quad v\in\{-c,c\}.
			\end{align}
			Equivalently, 
			$u(z,v,t)$ satisfies the following Kolmogorov backward equation
			\begin{align}\label{eq:ke}
			&\frac{\partial u }{\partial t}(z,v,t)  =\mathfrak  L u(z,v,t), \quad z=re^{i\theta},0<r<1, \theta\in[0,2\pi), v\in\{-c,c\},\\
			&u(0,z,v)=f(z,v),\quad z\in \overline{\mathbb D}, \notag
			\end{align}
			where  $\mathfrak L$ is the infinitesimal generator of $\{(Z(t),V(t)), t\geq 0\},$ given by 
			$$\mathfrak Lf(z,v)=izv f'(z,v)+\lambda[f(z,-v)-f(z,v)]$$
			where $f'(z,v)=\frac{\partial f(x,v)}{\partial x}\Big|_{x=z}$ and Dom$(\mathfrak L)=\{f\in B(\mathbb G,\mathbb C):z\mapsto f(z, v)\in A(\mathbb D)\}.$
			
			\item Let $R_{\mu}(\mathfrak L) f  = (\mu - \mathfrak L)^{-1}$ be  the resolvent of $ \mathfrak L $ and $ \rho(\mathfrak L) $ its resolvent set. If $  z\mapsto f(z, v)\in A(\mathbb D) $ and $\mu>0,$ then $z \mapsto R_{\mu}(\mathfrak L) f(z,v) \in A(\mathbb D)$ and the resolvent takes the form 
			\begin{align}
			R_{\mu}(\mathfrak L) f (z,v) = 
			\sum_{k=0}^\infty z^k[\tilde{d}_k(\mu,v) + \tilde{e}_k(\mu,v)],\quad z\in \mathbb D,
			\end{align}
			where
			\begin{align}
			\tilde{d}_k(\mu,v) = 	a_k(v) \bigg( - \sgn(v) \frac{1}{\lambda + \mu + i k c}  + \frac{\lambda + \mu + i k v }{\mu^2 + 2 \lambda \mu + c^2k^2}\bigg) 
			\end{align}
			and
			\begin{align}
			\tilde e_k(\mu, v) = 
			\frac{\lambda}{\mu^2 + 2 \lambda \mu + c^2k^2} a_k(-v)
			\end{align}
			
		\end{enumerate}
	\end{theorem}
	
	\begin{proof}
		(i)  It is sufficient to observe that
		$$T_t f(z,v)=\es_{z,v}[f(Z(t),V(t))]=\es_{z,v}[f_1(Z(t),V(t))]+i\es_{z,v}[f_2(Z(t),V(t))],$$
		where $f_1:=$Re$f\in B(\mathbb G, \mathbb R)$ and $f_2:=$ Im$f\in B(\mathbb G,\mathbb R).$ Therefore, by applying Lemma 19.3 in \cite{bass_2011}, we obtain $\es_{z,v}[f_1(Z(t),V(t))]$ and $\es_{z,v}[f_2(Z(t),V(t))]$
		belong to $B(\mathbb G, \mathbb R)$ which implies  $T_t f\in B(\mathbb G, \mathbb C).$
		
		(ii)  We recall the following results on the cosine and sine Fourier transforms of Bessel functions 
		\begin{equation}\label{eq:olomform1}
		\int_0^\infty \cos(\xi x)J_0(b\sqrt{a^2-x^2})\mathds 1_{0<x<a}\de x=\frac{\sin(a\sqrt{b^2+\xi^2})}{\sqrt{b^2+\xi^2}},\quad b\in\mathbb C,
		\end{equation}
		and
		\begin{equation}\label{eq:olomform2}
		\int_0^\infty \sin(\xi x) x (a^2-x^2)^{\nu/2}J_\nu(b\sqrt{a^2-x^2})\mathds 1_{0<x<a}\de x=2^{-1/2}\pi^{1/2} b^\nu a^{\nu+\frac32} \xi \frac{J_{\nu+\frac32}(a\sqrt{b^2+\xi^2})}{(b^2+\xi^2)^{\frac\nu 2+\frac34}} ,\quad b\in\mathbb C,
		\end{equation}
		see formula 6.738, at page 742 in \cite{GR}.
		
		We can write
		\begin{align}\label{eq:olom}
		&\es_{z,v}[f(Z(t),V(t))]=\es_v[f(ze^{iX(t)},V(t))]\\
		&=\int_{\mathbb R}f(ze^{ix},v)\mathds P(X(t)\in \de x, V(t)=v|V(0)=v)\notag\\
		&\quad+\int_{\mathbb R}f(ze^{ix},-v)\mathds P(X(t)\in \de x, V(t)=-v|V(0)=v).\notag\\
		&=\int_{\mathbb R}f(ze^{ix},v)\left[e^{-\lambda t}\delta(x-  vt) \de x+\frac{\lambda}{2c}(ct + x\,\text{sgn}(v)) \frac{I_1(\frac\lambda c\sqrt{c^2t^2-x^2})}{\sqrt{c^2t^2-x^2}}\mathds 1_{|x|<ct}\right]\de x\notag\\
		&\quad+\int_{\mathbb R}f(ze^{ix},-v)e^{-\lambda t}\frac{\lambda}{2c}I_0\left(\frac\lambda c\sqrt{c^2t^2-x^2}\right)\mathds 1_{|x|<ct}\de x.\notag
		\end{align}
		For $z\in\mathbb D,$ since $z\mapsto f(z,\pm v)$ is analytic in $\mathbb D,$ we have that $|f(ze^{ix},\pm v)|\leq \sum_{k=0}^\infty |a_k(\pm v) z^k| <\infty,$ which implies that $\sum_{k=0}^\infty a_k(\pm v) z^k e^{ikx}$ is uniformly convergent for any $x\in\mathbb R.$ Therefore, we are able to interchange the integral with the terms of the series as follows
		\begin{align}\label{eq:olombis}
		&\es_v[f(ze^{iX(t)},V(t))]\\
		&=\sum_{k=0}^\infty z^k a_k(v)e^{-\lambda t}\left[\int_{\mathbb R}e^{ikx} \delta(x-  vt) \de x+\frac{\lambda}{2c}\int_{\mathbb R}e^{ikx}(ct + x\,\text{sgn}(v)) \frac{I_1(\frac\lambda c\sqrt{c^2t^2-x^2})}{\sqrt{c^2t^2-x^2}}\mathds 1_{|x|<ct}\de x\right]\notag\\
		&\quad+\sum_{k=0}^\infty z^k a_k(-v)e^{-\lambda t}\frac{\lambda}{2c} \int_{\mathbb R}e^{ikx }I_0\left(\frac\lambda c\sqrt{c^2t^2-x^2}\right)\mathds 1_{|x|<ct}\de x\notag\\
		&=e^{-\lambda t}\sum_{k=0}^\infty z^k a_k(v) e^{ikvt} + \frac{\lambda t}{2} e^{-\lambda t} \sum_{k=0}^\infty z^k a_k(v)\int_{\mathbb R}e^{ikx } \frac{I_1(\frac\lambda c\sqrt{c^2t^2-x^2})}{\sqrt{c^2t^2-x^2}}\mathds 1_{|x|<ct}\de x\notag \\
		&\quad+\frac{\lambda}{2c}e^{-\lambda t}\text{sgn}(v)\sum_{k=0}^\infty z^k a_k(v)\int_{\mathbb R}xe^{ikx } \frac{I_1(\frac\lambda c\sqrt{c^2t^2-x^2})}{\sqrt{c^2t^2-x^2}}\mathds 1_{|x|<ct}\de x\notag\\
		&\quad+e^{-\lambda t}\frac{\lambda}{2c} \sum_{k=0}^\infty z^k a_k(-v)\int_{\mathbb R}e^{ikx }I_0\left(\frac\lambda c\sqrt{c^2t^2-x^2}\right)\mathds 1_{|x|<ct}\de x\notag\\
		&=e^{-\lambda t}\sum_{k=0}^\infty z^k a_k(v)e^{ikvt} +\frac{e^{-\lambda t}}{2c}\sum_{k=0}^\infty z^k a_k(v)\int_{\mathbb R}e^{ikx }\frac{\partial}{\partial t} I_0\left(\frac\lambda c\sqrt{c^2t^2-x^2}	\right)\mathds 1_{|x|<ct}\de x\notag\\
		&\quad+\frac{e^{-\lambda t}}{2c^2 t}\text{sgn}(v)\sum_{k=0}^\infty z^k a_k(v)\int_{\mathbb R}xe^{ikx } \frac{\partial}{\partial t} I_0\left(\frac\lambda c\sqrt{c^2t^2-x^2}	\right)\mathds 1_{|x|<ct}\de x\notag\\
		&\quad+e^{-\lambda t}\frac{\lambda}{2c} \sum_{k=0}^\infty z^k a_k(-v)\int_{\mathbb R}e^{ikx }I_0\left(\frac\lambda c\sqrt{c^2t^2-x^2}\right)\mathds 1_{|x|<ct}\de x.\notag
		\end{align}
		
		Now, by exploiting  \eqref{eq:olomform1} and \eqref{eq:olomform2}, we obtain the following results
		\begin{align}\label{eq:Int1}
		\int_{\mathbb R}e^{i k x} I_0\left(\frac\lambda c\sqrt{c^2t^2-x^2}\right)\mathds 1_{|x|<ct}\de x&=2\int_{0}^\infty \cos( k x) I_0\left(\frac\lambda c\sqrt{c^2t^2-x^2}\right)\mathds 1_{0<x<ct}\de x\\
		&=2\int_{0}^\infty \cos( kx) J_0\left(i\frac\lambda c\sqrt{c^2t^2-x^2}\right)\mathds 1_{0<x<ct}\de x\notag\\
		&=\frac{2c}{\sqrt{c^2k^2-\lambda^2}}\sin(t\sqrt{c^2k^2-\lambda^2}),\notag
		\end{align}

		\begin{align}\label{eq:Int2}
		\int_{\mathbb R}e^{ikx }\frac{\partial}{\partial t} I_0\left(\frac\lambda c\sqrt{c^2t^2-x^2}	\right)\mathds 1_{|x|<ct}\de x&=\int_{-ct}^{ct} e^{ikx }\frac{\partial}{\partial t}  I_0\left(\frac\lambda c\sqrt{c^2t^2-x^2}\right)\de x\\
		&= \frac{\partial}{\partial t}  \int_{-ct}^{ct} e^{ikx }I_0\left(\frac\lambda c\sqrt{c^2t^2-x^2}\right)\de x -2c\cos(kct)\notag\\
		&= \frac{\partial}{\partial t}  \int_{\mathbb R}e^{i xk} I_0\left(\frac\lambda c\sqrt{c^2t^2-x^2}\right)\mathds 1_{|x|<ct}\de x-2c\cos(kct)\notag\\
		&=2c[\cos(t\sqrt{c^2k^2-\lambda^2})- \cos(kct)],\notag
		\end{align}
		and
		\begin{align}\label{eq:Int3}
		&\int_{\mathbb R}xe^{i xk}\frac{\partial}{\partial t}  I_0\left(\frac\lambda c\sqrt{c^2t^2-x^2}\right)\mathds 1_{|x|<ct}\de x\\
		&=\int_{-ct}^{ct} xe^{ikx }\frac{\partial}{\partial t}  I_0\left(\frac\lambda c\sqrt{c^2t^2-x^2}\right)\de x\notag\\
		&= \frac{\partial}{\partial t}  \int_{-ct}^{ct} x e^{ikx }I_0\left(\frac\lambda c\sqrt{c^2t^2-x^2}\right)\de x -2c^2 t\cos(kct)\notag\\
		&=2i \frac{\partial}{\partial t}  \int_{0}^{\infty} x \sin (kx ) I_0\left(\frac\lambda c\sqrt{c^2t^2-x^2}\right)\mathds 1_{0<x<ct}\de x -2c^2 t\cos(kct)\notag\\
		&=2i \frac{\partial}{\partial t}  \int_{0}^{\infty} x \sin (kx ) J_0\left(i\frac\lambda c\sqrt{c^2t^2-x^2}\right)\mathds 1_{0<x<ct}\de x -2c^2 t\cos(kct)\notag\\
		&=\sqrt{2\pi} i \frac{c^3k}{(c^2k^2-\lambda^2)^{\frac34}}\frac{\partial}{\partial t} \left\{t^{\frac 32}J_{\frac32}(t\sqrt{c^2k^2-\lambda^2})\right\}-2c^2 t\cos(kct)\notag\\
		&=\sqrt{2\pi} i \frac{c^3k}{(c^2k^2-\lambda^2)^{\frac34}}\frac{\partial}{\partial t} \left\{t\sqrt{\frac{2}{\pi \sqrt{c^2k^2-\lambda^2} }}\left(\frac{\sin(t \sqrt{c^2k^2-\lambda^2})}{t\sqrt{c^2k^2-\lambda^2}}-\cos(t \sqrt{c^2k^2-\lambda^2})\right)\right\}\notag\\
		&\quad-2c^2 t\cos(kct)\notag\\
		&=2i \frac{c^3 k}{\sqrt{c^2k^2-\lambda^2}}t \sin(t \sqrt{c^2k^2-\lambda^2})-2c^2 t\cos(kct).\notag
		\end{align}
		
		By plugging into \eqref{eq:olombis} the results \eqref{eq:Int1}, \eqref{eq:Int2} and \eqref{eq:Int3}, we get the expression \eqref{eq:teoolomof}. Besides, since $z\mapsto f(z,v)$ is continuous on $\overline{\mathbb D},$ from \eqref{eq:olom} follows the continuity of  $z\mapsto T_tf(z,v)$ on $\overline{\mathbb D}.$

		(iii)  We prove that $u(z,v,t),v\in\{-c,c\},$ are the solutions of the integral equations \eqref{eq:inteq}. We observe that
		\begin{align*}
		u(z,v,t)=\es_{z,v}[f(Z(t),V(t))]=\es_{z,v}[f(Z(t),V(t))\mathds 1_{T_1<t}]+\es_{z,v}[f(Z(t),V(t))\mathds 1_{T_1\geq t}]
		\end{align*}
		Clearly, $\es_{z,v}[f(Z(t),V(t))\mathds 1_{T_1\geq t}]=e^{-\lambda t} f(ze^{ivt},v),$ while
		\begin{align*}
		\es_{z,v}[f(Z(t),V(t))\mathds1_{T_1< t}]&=\es_{z,v}[\es_{z,v}[f(Z(t),V(t))\mathds 1_{T_1< t}|\mathcal F_{T_1}]]\\
		&=\es_{z,v}[\mathds1_{T_1< t}\,\es_{Z_1,-v}[f(Z(t),V(t))]]\\
		&=\lambda \int_0^t e^{-\lambda s} u(ze^{ivs},-v,t-s)\de s.
		\end{align*}
		where $Z_1:=Z(T_1)=ze^{iT_1v}$ and $V(T_1)=-v.$ 
		
		Now, we show the result \eqref{eq:ke}. First of all, we prove that the family of linear operator $\{T_t f,t\geq 0\}$ from $ B(\mathbb G,\mathbb C)$ to $ B(\mathbb G,\mathbb C)$ represents a $C_0$-semigroup (see Definition at page 5 in \cite{gold2}). Note that $\sup\{||T_tf||: f\in B(\mathbb G,\mathbb C), ||f||\leq 1\}<\infty.$ It is easy to prove that $T_t ,t\geq 0,$ is a semigroup; i.e. $T_0 f=f$ and $T_s T_t f=T_{s+t }f,$  for all $f\in B(\mathbb G,\mathbb C)$ and all $s,t\geq 0.$  Finally, we observe that
		\begin{align*}
		\es_{z,v}[f(Z(t),V(t))]&=(1-\lambda t) f(ze^{ivt},v)+\lambda t f(ze^{ivO(t)},-v)+o(t),
		\end{align*}
		which allows to claim that $t\mapsto T_t f$ is strongly continuous; that is
		\begin{equation}
		\lim_{t\downarrow 0}|| T_t f(z,v)-f(z,v)||=0,\quad f\in B(\mathbb G, \mathbb C).
		\end{equation}
		
		Let $\Psi (\theta, v):=f(re^{i\theta},v),$ where $0<r<1,\theta\in[0,2\pi)$ and $ v\in\{-c,c\}.$ Let $z\mapsto f(z,v)\in A(\mathbb D),$ from \eqref{eq:inteq} we derive
		\begin{align*}
		\es_{z,v}[f(Z(t),V(t))]&=(1-\lambda t) f(ze^{ivt},v)+\lambda t f(ze^{ivO(t)},-v)+o(t)\\
		&=(1-\lambda t) \Psi(\theta+vt,v)+\lambda t \Psi(\theta+vO(t),-v)+o(t)\\
		&=(1-\lambda t)\left[\Psi(\theta,v)+vt  \frac{\partial \Psi}{\partial \theta}(\theta,v)\right]+\lambda t \Psi(\theta+vO(t),-v)+o(t)\\
		&=(1-\lambda t)\left[f(z,v)+vt  \frac{\partial f}{\partial \theta}(z,v)\right]+\lambda t f(ze^{ivO(t)},-v)+o(t),
		\end{align*}
		as $t\downarrow 0.$
		Therefore, we have
		\begin{align*}
		\lim_{t\downarrow 0}\left|\left|\frac{\es_{z,v}[f(Z(t),V(t))]-f(z,v)}{t} -\mathfrak Lf(z,v)\right|\right|=0
		\end{align*}
		where
		\begin{align*}
		\mathfrak Lf(z,v)&=\mathfrak L\Psi(\theta,v)\\
		&=v  \frac{\partial \Psi}{\partial \theta}(\theta,v) +\lambda[\Psi(\theta,-v)-\Psi(\theta,v)], \\
		&= v  \frac{\partial f}{\partial \theta}(z,v) +\lambda[f(z,-v)-f(z,v)],\\
		&=izv f'(z,v)+\lambda[f(z,-v)-f(z,v)],\quad v\in \{-c,c\},
		\end{align*}
		where the convergence being with respect to the supremum norm $||f||=\sup_{(z,v)\in\mathbb G}|f(z,v)|.$  Since $T_tf ,t\geq 0,$ is a $C_0$-semigroup, by Theorem 1.2 in \cite{gold2} follows the result \eqref{eq:ke}.
		
		(iv) We recall that the resolvent set $\rho(\mathfrak L)$ is given by $\mu\in\mathbb R$ such that $\mu-\mathfrak L:$Dom$(\mathfrak L)\to B(\mathbb G,\mathbb C)$ is bijective and $R_\mu(\mathfrak L)$ is a bounded linear operator from Dom$(\mathfrak L)$ to $B(\mathbb G,\mathbb C).$  Since $||T_t||\leq 1$, for all $t\geq 0,$   and by resorting (iii), we have that $\{T_t f,t\geq 0\}$ is a $C_0$-contraction semigroup with generator $\mathfrak L.$ Then it is well-known (see Proposition 2.1 in \cite{ethier}) that $\rho(\mathfrak L)\neq \varnothing;$ i.e. $(0,\infty)\subset \rho(\mathfrak L),$ and
		\begin{align}\label{eq:res}
		R_\mu(\mathfrak L) f = \int_0^\infty e^{- \mu t } T_t f \de t
		\end{align}
		for all $f\in  B(\mathbb G,\mathbb C)$ and $\mu>0.$ The representation \eqref{eq:res} implies that $R_\mu(\mathfrak L)\in $ Dom$(\mathfrak L)$ for $\mu>0$ and the first claim of the theorem immediately follows.

		For the second part of the statement, note that by recalling (ii), we have that 
		its resolvent can be expressed as, for $ s>0 $ (see, e.g., \cite{GR}),
		\begin{align}
		R_s(\mathfrak L) f(z,v) = \int_0^\infty e^{- s t } T_t f(z,v) \de t = \int_0^\infty e^{-s t} e^{-\lambda t} 
		\sum_{k=0}^\infty z^k [ d_k(t,v) + e_k(t,v)] \de t
		\end{align}
		Let $ \mathcal L_s (f) $ be the Laplace transform of $ f$ evaluated at point $ s $. We recall the well-known relationships
		\begin{gather*}
		\mathcal L_s (\cos(at)) = \frac{s}{s^2 + a^2}, \quad 
		\mathcal L_s (\sin(at)) = \frac{a}{s^2 + a^2} , \quad
		\mathcal L_s (\cosh(at)) = \frac{s}{s^2 - a^2} , \quad\\
		\mathcal L_s (\sinh(at)) = \frac{a}{s^2 - a^2} , \quad
		\mathcal L_s (e^{iat}) = \frac{s + i a}{s^2 + a^2}. \quad
		\end{gather*}
		Then 
		\begin{align*}
		& \int_0^\infty 
		e^{-(\lambda + s)t} d_k(t,v) \de t  
		\\
		&= 
		a_k(v) \int_0^\infty e^{- (\lambda + s )t} \Big[ 
		e^{ikvt}+\cos(t\sqrt{c^2k^2-\lambda^2}) 
		+ i \text{\sgn}(v)\frac{ck \sin(t\sqrt{c^2k^2-\lambda^2})}{\sqrt{c^2k^2-\lambda^2}}
		\notag \\
		& \quad -(1+\text{\sgn}(v))\cos(kct)
		\Big] \de t
		\notag \\
		&=
		a_k(v)\bigg( \frac{\lambda + s + i k v}{(\lambda + s)^2 + (kv)^2} 
		+ \frac{\lambda + s}{(\lambda + s)^2 + c^2k^2 - \lambda^2} 
		+\frac{i\,  \sgn (v) ck}{(\lambda + s)^2 + c^2k^2 - \lambda^2}
		\notag \\
		& \quad -(1 + \sgn (v)) \frac{\lambda + s}{(\lambda + s)^2 + c^2k^2} \bigg) 
		\notag \\
		&=
		a_k(v) \bigg( - \sgn(v) \frac{1}{\lambda + s + i k c}  + \frac{\lambda + s + i k v }{s^2 + 2 \lambda s + c^2k^2}\bigg)  
		\notag 
		\end{align*}
		and
		\begin{align*}
		\int_0^\infty 
		e^{-(\lambda + s)t} e_k(t,v) \de t  
		&= 
		a_k(-v) \int_0^\infty e^{- (\lambda + s )t} 
		\frac{\lambda \sin(t\sqrt{c^2k^2-\lambda^2})}{\sqrt{c^2k^2-\lambda^2}}\de t
		\\
		&= 
		\frac{\lambda}{s^2 + 2 \lambda s + c^2k^2} a_k(-v),
		\notag 
		\end{align*}
		which concludes the proof.
	\end{proof}
	
	\begin{remark}
	It is worth to observe that:
	\begin{itemize}
	\item in the theory of complex evolution equations the space of functions $A(\mathbb D)$ assumes the same  role of the space of all bounded uniformly continuous function on $\mathbb R$ which arises in the study of the semigroup operators. On this point see \cite{GGG2};
	\item the result \eqref{eq:ke} can be rewritten as follows. Let $\hat f:=(f(z,c), f(z,-c))'$ where $f\in$Dom($\mathfrak L$); one has that infinitesimal generator of $ (X(t),V(t),t\geq 0)$ is given by
\end{itemize}
	$$\mathfrak L\hat f=\left[ \left(  \begin{matrix} 
	      i zc\frac{\partial}{\partial z} & 0 \\
	      0 & -iz c\frac{\partial}{\partial z}  \\
	   \end{matrix}\right)+ \left(  \begin{matrix} 
	     -\lambda & \lambda \\
	      \lambda & -\lambda  \\
	   \end{matrix}\right) \right]\hat f.$$
	   Therefore, the infinitesimal generator of the circular telegraph process can be obtained form that of the standard telegraph process by replacing the velocity $\pm c$ (on the line) with  complex velocity $\pm iz c.$	\end{remark}
	The next result regards a complex version of the telegraph equation and connects the wrapped up telegraph process with the hyperbolic partial differential equations.
	
	\begin{corollary}
		Let  $u_1:=u(z,c,t)$ and $u_2:=u(z,-c,t).$ We denote by  $ p:=\frac{u_1+u_2}{2}$ and $ w:=\frac{u_1-u_2}{2}.$ Then, $p$ is solution to the complex telegraph equation
		\begin{equation}\label{eq:cte}
		\frac{\partial^2 p}{\partial t^2}+2\lambda \frac{\partial p}{\partial t}=-c^2 z^2[p''+p']=c^2\frac{\partial^2 p}{\partial\theta^2},\quad z=re^{i\theta}, z\neq 0, \theta\in[0,2\pi), 0<r<1.
		\end{equation}
	\end{corollary}
	
	\begin{proof}
		From \eqref{eq:ke}, we can write down the following system
		\begin{equation}
		\begin{cases}
		\frac{\partial u_1}{\partial t}=iz c u_1'+\lambda[u_2-u_1]\\
		\frac{\partial u_2}{\partial t}=-iz c u_2'+\lambda[u_1-u_2]
		\end{cases}
		\end{equation}
		and then
		\begin{equation}
		\begin{cases}
		\frac{\partial p}{\partial t}=iz c w'\\
		\frac{\partial w}{\partial t}=izcp'-2\lambda w
		\end{cases}
		\end{equation}
		
		Therefore
		$$\frac{\partial^2 p}{\partial t^2}=iz c \frac{\partial w'}{\partial t},\quad \frac{\partial^2 w}{\partial t\partial \theta}=iz \frac{\partial w'}{\partial t}=(iz)^2 c p'+(iz)^2 c p''-2\lambda izw'$$
		
		Hence
		$$\frac{\partial^2 p}{\partial t^2}=-c^2 z^2[p''+p'] -2\lambda izc w'=-c^2 z^2[p''+p'] -2\lambda \frac{\partial p}{\partial t}$$
		
	\end{proof}

	\begin{remark}
		As shown in \cite{GGG}, one has that
		\begin{align}\label{eq:lawtelcomp}
		u(z,t)
		&=\frac12 e^{-\lambda t}[f(ze^{ ict})+f(ze^{ -ict})]\\
		&\quad+\frac{\lambda e^{-\lambda t}}{2}  \int_{-t}^{t} I_0\left(\lambda \sqrt {t^2-x^2}\right)f(ze^{ icx})\de x+\frac{\lambda t e^{-\lambda t}}{2}   \int_{-t}^{t} \frac{ I_1\left(\lambda \sqrt{t^2-x^2}\right)}{\sqrt{t^2-x^2}}f(ze^{ icx})\de x\notag\\
		&=\frac12 e^{-\lambda t}[f(ze^{ ict})+f(ze^{ -ict})]\notag\\
		&\quad+\frac{\lambda e^{-\lambda t}}{2c}  \int_{-ct}^{ct} I_0\left(\frac{\lambda}{c} \sqrt {c^2t^2-x^2}\right)f(ze^{ ix})\de x+\frac{\lambda t e^{-\lambda t}}{2}   \int_{-ct}^{ct} \frac{ I_1\left(\frac{\lambda}{c} \sqrt{c^2t^2-x^2}\right)}{\sqrt{c^2t^2-x^2}}f(ze^{ ix})\de x\notag\\
		&=\mathds E[f (ze^{iX(t)})]\notag
		\end{align}
		is the solution of \eqref{eq:cte} subjected to the initial conditions of $u(z,0)=f(z), \quad \frac{\partial u}{\partial t}(z,0)=0$. Furthermore, similarly to the Kac's interpretation of the solution of the classical telegraph equation, \eqref{eq:lawtelcomp} has the following stochastic representation 
		
		\begin{equation}
		u(z,t)=\frac12 \mathds{E}\left[f\left(ze^{ic\int_0^t (-1)^{N(s)} \de s}\right)\right]+\frac12 \mathds{E}\left[f\left(ze^{-ic\int_0^t (-1)^{N(s)} \de s}\right)\right]
		\end{equation}
		where $\mathds{E}$ represents the mean value w.r.t. $N(t), t\geq 0.$ 
	\end{remark}

	\section{On the probability distribution of $Z$ }

	For the sake of simplicity we deal with telegraph random evolution  on a unit circle $\mathbb S:=\mathbb S_1$ and $Z(0)=1,$ i.e. the initial angle is $\theta=0$. In order to study the probability distribution of $Z,$ a classical approach (see, e.g., \cite{feller}) is based on the identification of the two endpoints of the interval $[0,2\pi)$ and then the latter is interpreted as a circle of unit radius: i.e we apply the quotient map $x\mapsto e^{ix}$ or $x\mapsto x (\text{mod}\, 2\pi)$,  $x\in [0,2\pi)$. In other words, we identify $Z$ with its stochastic angular component $\{X(t),t\geq 0\}$.
	
	
	Now, let  $p$ and $\mu_n$ be the probability laws of the classical telegraph given by \eqref{eq:lawtel} and \eqref{eq:disttp}, respectively. For any $t\geq 0,$ the circular telegraph process has the following wrapped probability laws (see, e.g., \cite{mardia} for an introduction to the wrapped probability distributions)
	\begin{align}\label{eq:densctp}
	p_Z(\de \theta, t) &=\mathds P(Z(t)\in \de \theta)\\
	&=\sum_{k=-\infty }^\infty \mathds P(X(t)\in  \de(\theta+2k\pi))\notag\\
	&=\sum_{k=-\infty }^\infty p(\de(\theta+2k\pi),t)\notag\\
	&=
	\sum_{k=-\infty }^\infty
	\frac12 e^{-\lambda t}[\delta_{- ct - 2 k \pi}(\de \theta)+\delta_{ct - 2 k \pi}(\de \theta)]+
	\sum_{k=-\infty }^\infty
	\mu(\theta + 2 k \pi,t) \mathds{1}_{|\theta + 2 k \pi |<ct}\,\de \theta,\notag
	\end{align}
	while its conditional probability distribution with respect to the number of changes of rotation becomes
	\begin{equation}
	p_Z^n( \de\theta, t)=\sum_{k=-\infty}^\infty\mu_n(\theta+2k\pi,t)	\de \theta,
	\end{equation}
	where $\theta\in[0,2\pi)$. It is worth to mention that $\frac12 e^{-\lambda t}\sum_{k=-\infty }^\infty[\delta_{- ct - 2 k \pi}(\de \theta)+\delta_{ct - 2 k \pi}(\de \theta)]$ represents the singular component of the probability distribution of $Z$ and it emerges from the following 
	$$\mathds P(Z(t)=e^{\pm ict})=\frac12 \mathds P(N(t)=0)=\frac12 e^{-\lambda t}.$$
	Clearly, the second part in \eqref{eq:densctp} is the absolutely continuous component of the probability law of $Z(t), t\geq 0.$
	
	\begin{remark}
	From \eqref{eq:densctp}, it is easy to check that the centripetal acceleration $A(t)=-c^2 Z(t), t\geq 0, (\mathds P$-a.e.) admits as absolutely continuous component of its probability law
	$$\frac{1}{c^2}\sum_{k=-\infty }^\infty
	\mu(-\varphi /c^2 + 2 k \pi,t) \mathds{1}_{|-\varphi/c^2 + 2 k \pi |<ct},\quad \varphi \in(-2\pi c^2,0]. $$
	\end{remark}
	
	The following theorem concerns the moments of the circular telegraph process $\{Z(t),t\geq 0\}.$

	\begin{theorem}\label{th:mom}
		We have
		\begin{align}\label{eq:alfamom}
		\mathds{E}[Z^\alpha(t) ]&=\mathds{E}[e^{i\alpha X(t)}]\\
		&= e^{-\lambda t}\left[\cosh(t\sqrt{\lambda^2-\alpha^2c^2})+\frac{\lambda }{\sqrt{\lambda^2-\alpha^2c^2}}\sinh(t\sqrt{\lambda^2-\alpha^2c^2})\right],\quad \alpha\in\mathbb R,\notag
		\end{align}
		and
		\begin{align}\label{eq: cov}
		\mathds{E}\left[Z^\alpha(t) Z^\beta(t') \right]&=e^{-\lambda \max\{t,t'\}}\Big[\cosh(\min\{t,t'\}\sqrt{\lambda^2-(\alpha + \beta)^2c^2}) 
		\\
		& 
		\quad+\frac{\lambda }{\sqrt{\lambda^2- (\alpha + \beta)^2 c^2}}\sinh(\min\{t,t'\}\sqrt{\lambda^2-(\alpha + \beta)^2 c^2})\Big] \times\notag \\
		&\quad \times\left[\cosh(|t- t'|\sqrt{\lambda^2- \omega^2 c^2  })+\frac{\lambda }{\sqrt{\lambda^2- \omega ^2 c^2}}\sinh(|t- t'|\sqrt{\lambda^2- \omega ^ 2 c^2})\right],\notag
		\end{align}
		where $\omega  = \alpha \, \mathds{1}_{t > t'} + \beta \,\mathds{1}_{t \leq t'}, \, \alpha, \beta \in \mathbb{R}$.
	\end{theorem}
	\begin{proof}
		The result \eqref{eq:alfamom} follows by observing that the moment of order $\alpha$ of $Z(t)$ coincides with the characteristic function of the standard telegraph process \eqref{eq:cftp}.
		
		Let $t>t'.$ Therefore, we can write down
		\begin{align*}
		&\mathds{E}\left[Z^\alpha(t) Z^\beta(t')\right]\\
		&=\mathds{E}\left[\exp \left\{iV(0)\left(\alpha \int_0^t (-1)^{N(s)}\de s+ \beta \int_0^{t '}(-1)^{N(s)}\de s\right)\right\}\right]\\
		&=\mathds{E}\left[\exp\left\{i (\alpha + \beta )V(0)\int_0^{t'} (-1)^{N(s)}\de s+i\alpha V(0)\int_{t '}^t(-1)^{N(s)}\de s\right\}\right]\\
		&=\mathds{E}\left[\exp\left\{i (\alpha + \beta )V(0)\int_0^{t'} (-1)^{N(s)}\de s+i \alpha V(0)(-1)^{N(t')}\int_{t '}^t(-1)^{N(s)-N(t')}\de s\right\}\right]\\
		&=\mathds{E}\left[\exp\left\{i(\alpha + \beta )V(0)\int_0^{t'} (-1)^{N(s)}\de s+i\alpha V(0)(-1)^{N(t')}\int_{0}^{t-t'}(-1)^{N(u+t')-N(t')}\de u\right\}\right]\\
		&=\mathds{E}\left[\mathds{E}\left[\exp\left\{i(\alpha + \beta )V(0)\int_0^{t'} (-1)^{N(s)}\de s+i\alpha V(0)(-1)^{N(t')}\int_{0}^{t-t'}(-1)^{N(u+t')-N(t')}\de u\right\}\Bigg| N(t')\right]\right]\\
		&=\mathds{E}\left[\exp\left\{i(\alpha + \beta )V(0)\int_0^{t'} (-1)^{N(s)}\de s\right\}\right]\\
		&\quad\times \mathds{E}\left[\mathds{E}\left[\exp\left\{i\alpha V(0)(-1)^{N(t')}\int_{0}^{t-t'}(-1)^{N(u+t')-N(t')}\de u\right\}\Bigg| N(t')\right]\right]\\
		&=\mathds{E}\left[e^{i(\alpha + \beta )X(t')}\right] \mathds{E}\left[\mathds{E}\left[\exp\left\{i\alpha V(0)(-1)^{N(t')}\int_{0}^{t-t'}(-1)^{N(u)}\de u\right\}\Bigg| N(t')\right]\right]\\
		&=e^{-\lambda t'}\left[\cosh(t'\sqrt{\lambda^2-(\alpha + \beta)^2c^2})+\frac{\lambda }{\sqrt{\lambda^2-(\alpha + \beta)^2 c^2}}\sinh(t'\sqrt{\lambda^2-(\alpha + \beta)^2 c^2})\right]\\
		&\quad \times e^{-\lambda (t- t')}\left[\cosh((t-t')\sqrt{\lambda^2-\alpha^2 c^2})+\frac{\lambda }{\sqrt{\lambda^2-\alpha^2 c^2}}\sinh((t-t')\sqrt{\lambda^2-\alpha^2 c^2})\right]
		\end{align*}
		Therefore we can conclude that the result \eqref{eq: cov} holds true.\end{proof}

	\begin{remark}
		As noted in \cite{mardia}, for integer values of $ \alpha, $ \eqref{eq:alfamom} 
		can be interpreted as the characteristic function of a wrapped random variables.
		That is the Fourier transform of $ Z(t) $ is determined by the doubly infinite sequence
		\begin{equation}\label{eq:alfamom2}
		\phi_k^t = \int_0^{2\pi} e^{i k \theta } p_Z(\de \theta,t) = \mathds E e^{i k X(t)}, \quad k=0,\pm 1,\pm 2,.... 
		\end{equation}
		Analogously \eqref{eq: cov} gives the joint characteristic function of $(Z(t), Z(t'))$ for integer $\alpha, \beta  $. 
	\end{remark}

	\begin{corollary}
		The distribution of $Z(t)$, for any $ t \geq 0 $, admits the following representation
		\begin{align}\label{eq:fourier-full}
		&\int_A p_Z(\de \theta) 
		\\
		&= 
		\lim_{r\to 1^-} 
		\frac{1}{2\pi}\int_A 
		\bigg( 1 + 2 e^{-\lambda t} \sum_{k=1}^\infty 
		\left[\cos (t\sqrt{ k^2 c^2-\lambda^2}) + 
		\frac{ \lambda}{\sqrt{ k^2 c^2-\lambda^2}} \sin(t\sqrt{ k^2 c^2-\lambda^2}) \right] r^k \cos (k\theta) \bigg) \de \theta, \notag 
		\end{align}
		for all Borel subset $ A$ of $[0,2\pi)$ such that $ \partial A $ has null $ p_Z $-measure. 
	\end{corollary}
	\begin{proof}
		Following Th. XIX.6.1 in \cite{feller},  the distribution of $ Z $ can be recovered from its Fourier coefficients $ \phi_k^t $ as 
		\[
		\int_A p_Z(\de \theta) 
		= 
		\lim_{r\to 1^-} 
		\int_A 
		\sum_{k=-\infty}^\infty 
		\phi_k^t r^{|k|} e^{i k \theta}  \de \theta,
		\]
		where the result \eqref{eq:fourier-full} emerges after trivial algebra and by exploiting \eqref{eq:alfamom2}.
		Note that the series $  \sum \phi_k^t e^{i k \theta} $ is not guaranteed to converge. 
	\end{proof}

	\begin{remark}
		
		From Theorem \ref{th:mom}, we obtain the following covariance-type function of the circular telegraph process
		\begin{align*}
		\Gamma_Z(t,t')&=\mathds{E}\left[Z(t) Z(t')\right]- \mathds{E}\left[Z(t)\right]\mathds{E}\left[ Z(t')\right]\\
		&=e^{-\lambda \max\{t,t'\}}\left[\cosh(\min\{t,t'\}\sqrt{\lambda^2-4c^2})+\frac{\lambda }{\sqrt{\lambda^2-4c^2}}\sinh(\min\{t,t'\}\sqrt{\lambda^2-4c^2})\right]\\
		&\quad \times\left[\cosh(|t- t'|\sqrt{\lambda^2-c^2})+\frac{\lambda }{\sqrt{\lambda^2-c^2}}\sinh(|t- t'|\sqrt{\lambda^2-c^2})\right]\\
		&\quad-e^{-\lambda (t+t')}\left[\cosh(t\sqrt{\lambda^2-c^2})+\frac{\lambda }{\sqrt{\lambda^2-c^2}}\sinh(t\sqrt{\lambda^2-c^2})\right]\\
		&\quad\times\left[\cosh(t'\sqrt{\lambda^2-c^2})+\frac{\lambda }{\sqrt{\lambda^2-\alpha^2c^2}}\sinh(t'\sqrt{\lambda^2-c^2})\right].
		\end{align*}
	\end{remark}
	
	\section{On the limit process}
	
	Let us deal with the rescaled circular process $Y_n:=\{(Y_n(t),t\geq 0\}, n\in\mathbb N,$ on $\mathbb S,$ where 
	$$Y_n(t):=(Z({nt}))^{\sigma_n}=ze^{i \sigma_n X(nt)},\quad t\in[0,1], Z(0)=z^{1/\sigma_n},$$
	and $\sigma_n:=\frac{1}{c}\sqrt{\frac \lambda n}.$
	In the next theorem, we present the results on the asymptotic behavior of $Y_n.$ Let $C([0,1], \mathbb S)$ be the space of all $\mathbb S$-valued continuous functions on $[0,1]$ equipped
	with the sup-norm topology and $\Longrightarrow$ stands for the weak convergence in $C([0,1], \mathbb S)$.
	
	\begin{theorem}\label{th:conv}
		We have that:
		
		(i) (strong law of large numbers) $(Z(t))^{t^{-1/p}}\longrightarrow 1$ a.s., for $1\leq p<2,$ as $t\to\infty;$
		
		(ii) (weak convergence)
		$Y_n\Longrightarrow z \mathfrak B,$
		where $\mathfrak B:=\{\mathfrak B(t),t\geq 0\}$ represents the circular Brownian motion on $\mathbb S;$ i.e. $\mathfrak B(t):=e^{i B(t)},$  $B:=\{B(t),t\geq 0\}$ is a standard Brownian motion.
	\end{theorem} 
	\begin{proof}
		The points (i) immediately follows from Theorem 1.4 in \cite{hor} or Theorem 2.1 in \cite{ghosh} which state that $||X(t)||=o(t^{1/p})$. The result (ii) is an application of Theorem 1.2 \cite{hor} or Theorem 2.3 in \cite{ghosh} and of the continuous mapping theorem.
		In fact, if 
		\[
		X_n(t) :=\sigma_nX(nt), \quad t\in[0,1],
		\] 
		one has that $ X_n$ weakly converges to $B$ in  $C([0,1], \mathbb R)$ and by the continuous mapping theorem, we get 
		\[
		Y_n(t) =ze^{i X_n(t)}  \Longrightarrow z\mathfrak B(t), \quad  t \in [0,1].
		\]\end{proof}
	
	We observe that the density function of the circular Brownian motion $\mathfrak B(t)$, for a fixed $t>0,$ is given by
	\begin{align*}
	\nu(\theta,t)=\frac{1}{\sqrt{2\pi t}}\sum_{k =-\infty}^\infty e^{(\theta+2k\pi)^2/2t},\quad \theta\in[0,2\pi),
	\end{align*}
	which has the following representation in terms of Fourier series
	\begin{align*}
	\nu(\theta,t)=\frac{1}{2\pi }\left(1+2\sum_{k=1}^\infty e^{-k^2 t/2}\cos (k\theta)\right),\quad \theta\in[0,2\pi),
	\end{align*}
	see, e.g., \cite{steph}, \cite{hart} and \cite{toaldo}. 	Moreover 
	$$W_t f(z)=\mathds E f(z\mathfrak B(t))=\frac{1}{\sqrt{2\pi t}}\int_{-\infty}^{+\infty} f(ze^{i x})e^{-\frac{x^2}{2t}}\de x,$$
	where $f\in A(\mathbb D),$ is a $C_0$-semigroup of linear operators on $A(\mathbb D)$ and  represents the unique solution of the Cauchy problem
	\begin{equation}\label{eq:heatcirc}
	\begin{cases}
	\frac{\partial u}{\partial t}(z,t)=\frac12 \frac{\partial^2 u}{\partial \theta}(z,t),& (z,t)\in \mathbb D\setminus \{0\}\times [0,\infty),\\
	u(z,0)=f(z),& z\in \overline{ \mathbb D}.
	\end{cases}
	\end{equation}
	\begin{remark}
		For the Brownian motion on a circle $\mathfrak B$ the following properties hold:
		\begin{itemize}
			\item for $t>t',$ we have
			\begin{align*}
			\mathds{E}[\mathfrak B(t)^\alpha\mathfrak B(t')^\beta]
			&=\mathds{E}[e^{i\alpha(B(t)-B(t')+i(\alpha+\beta) B(t')}]\\
			&=e^{-\frac{\alpha^2(t-t')^2+(\alpha+\beta)^2t'}{2}}.
			\end{align*}
			Therefore
			\begin{align*}
			\mathds{E}[\mathfrak B(t)^\alpha\mathfrak B(t')^\beta]
			=\exp\left\{-\frac{[\omega(t-t')]^2+(\alpha+\beta)^2\min\{t,t'\}}{2}\right\},
			\end{align*}
			where $\omega  = \alpha \, \mathds{1}_{t > t'} + \beta \,\mathds{1}_{t \leq t'}, \, \alpha, \beta \in \mathbb{R};$
			\item the circular Brownian motion process  is not a martingale w.r.t. $\mathcal G_t:=\sigma(B(s),s\leq t);$ i.e.
			\begin{align*}
			\mathds{E}[\mathfrak B(t)|\mathcal G_s]&=\mathds{E}[\mathfrak B(t)|\mathcal G_s]\\
			&=\mathds{E}[e^{i(B(t)-B(s))}e^{i B(s)}|\mathcal G_s]\\
			&=\mathfrak B(s) e^{-\frac{(t-s)^2}{2}}.
			\end{align*}
		\end{itemize}
	\end{remark}
	\begin{remark}
		Under the Kac conditions; i.e. $\lambda, c\to\infty$ such that $c^2/\lambda\to 1,$ we have immediately the complex telegraph equation \eqref{eq:telpde} tends to the heat equation in \eqref{eq:heatcirc}
		which admits as solution $W_t f(z),$ or, equivalently, under the Kac conditions, the solution of \eqref{eq:telpde} tends to $W_t f(z);$ i.e. $\mathds E f(z e^{i X(t)})\longrightarrow \mathds E f(z \mathfrak B(t))$ (see also Theorem 2.6 in \cite{kolrat}).
	\end{remark}

	\section{Random harmonic oscillators}
	
	Let $Z_1(t):=\cos X(t),$ with $Z_1(0)=1,$ and $Z_2(t):= \sin X(t),$ with $Z_2(0)=0.$ We observe that $\{Z_1(t),\,t\geq 0\}$ and $\{Z_2(t),\,t\geq 0\}$ represent the projections of the wrapped up telegraph process $Z(t)$ on the $x$-axis and $y$-axis, respectively. They can be interpreted as models for a randomized version of a harmonic oscillator. In fact, the vertical component of the wrapped telegraph process $ Z_2(t)$
	coincides with the dynamics of a harmonic oscillator experiencing random collisions occurring at Poisson times, and $ Z_1(t)$ with its random velocity. At each Poisson event the direction of the motion is reversed, preserving the modulus of the velocity. That is, if the oscillator is at position $Z_2(t)$ with velocity $ \dot Z_2(t)=V(t) \cos X(t)$ and 
	a Poisson event occurs, then the oscillator continues its motion from position $Z_2(t)$ with new velocity $ -\dot Z_2(t)$.
	A sample path of $ Z_1(t) $ and $ Z_2(t) $ are shown in \autoref{fig:circle-path}.

Let $p(\de x,t)$ and $\mu(x,t)$ given by  \eqref{eq:lawtel}  and \eqref{eq:acpart}, respectively.	Let $\mathds F_X(x,t):=\int_{-\infty}^x p(\de y,t)$ be the cumulative distribution function of $X(t)$ for each $t\geq 0,$ We are able to derive the probability laws of  $Z_1(t)$ and $Z_2(t), t\geq 0.$
	
	\begin{theorem}
		1) Let $t\geq 0,$
		the cumulative distribution function of $Z_1(t)$ is given by
		\begin{align}\label{eq:cdfz1}
		\mathds P(Z_1(t)\leq x)=\sum_{k=-\infty}^{\infty}\left[\mathds F_X((2k+2)\pi- \arccos x,t)-\mathds F_X(\arccos x+ 2k\pi,t)\right],\quad |x|\leq 1.
		\end{align}
		Furthermore, 
		\begin{align*}
		\mathds P(Z_1(t)\in \de x)=e^{-\lambda t}
		\delta_{\cos ct}(\de x)+\nu(x,t)\mathds1_{|x|<1}\de x,
		\end{align*}
		where $\nu(x,t)\mathds1_{|x|<1}$ represents the density function of the absolutely continuous part of $Z_1(t)$ and reads
		\begin{align}\label{eq:densz1}
		\nu(x,t)&=\frac{1}{\sqrt{1-x^2}}\sum_{k=-\infty}^{\infty}[\mu((2k+2)\pi- \arccos x,t)\,\mathds 1_{|(2k+2)\pi- \arccos x|<ct}\\
		&\quad+\mu(\arccos x+2k\pi,t)\,\mathds1_{| \arccos x+2k\pi|<ct}].\notag
		\end{align}
		
		2) Let $t\geq 0,$ the cumulative distribution function of $Z_2(t)$ becomes
		\begin{align}\label{eq:cdfz2}
		\mathds P(Z_2(t)\leq x)&=\sum_{k=-\infty}^{\infty}\left[\mathds F_X( \arcsin x+2k\pi,t)-\mathds F_X((2k-1/2)\pi,t)\right.\\
		&\quad\left.+\mathds F_X( (2k+3/2)\pi,t)-\mathds F_X( (2k+1)\pi-\arcsin x,t)\right],\quad |x|\leq 1.\notag
		\end{align}
		Furthermore,
		\begin{align}\label{eq:cdfz1}
		\mathds  P(Z_2(t)\in\de x)=\frac12e^{-\lambda t}[\delta_{\sin ct}(\de x)+\delta_{-\sin ct}(\de x)]+ \eta(x,t)\mathds 1_{|x|<1}\de x,
		\end{align}
		where $\eta(x,t)\mathds1_{|x|<1}$ is he density function of the absolutely continuous part of $Z_2(t),$ that is 
		\begin{align}\label{eq:densz2}
		\eta(x,t)&=\frac{1}{\sqrt{1-x^2}}\sum_{k=-\infty}^{\infty}[\mu(\arcsin x+2k\pi,t)\,\mathds1_{|\arcsin x+2k\pi|<ct}\\
		&\quad+\mu( (2k+1)\pi-\arcsin x,t)\,\mathds 1_{| (2k+1)\pi-\arcsin x|<ct}].\notag
		\end{align}
		
	\end{theorem}
	\begin{proof}
		1) Let us fix $|x|< 1.$ We have 
		\begin{align*}
		\mathds P(Z_1(t)\leq x)&=\sum_{k=-\infty}^{\infty}P(\cos X(t)\leq x, 2k\pi\leq X(t)\leq (2k+1)\pi)\\
		&\quad+\sum_{k=-\infty}^{\infty}\mathds P(\cos X(t)\leq x, (2k+1)\pi\leq X(t)\leq (2k+2)\pi)\\
		&=\sum_{k=-\infty}^{\infty}\mathds P( X(t)\geq \arccos x+2k\pi , 2k\pi\leq X(t)\leq (2k+1)\pi)\\
		&\quad+\sum_{k=-\infty}^{\infty}\mathds P( X(t)\leq(2k+2)\pi- \arccos x, (2k+1)\pi\leq X(t)\leq (2k+2)\pi)\\
		&=\sum_{k=-\infty}^{\infty}[\mathds P\left(\arccos x+ 2k\pi\leq X(t)\leq (2k+1)\pi\right)\\
		&\quad+\mathds P\left((2k+1)\pi\leq X(t)\leq(2k+2)\pi- \arccos x\right)]\\
		&=\sum_{k=-\infty}^{\infty}\mathds P\left(\arccos x+ 2k\pi\leq X(t)\leq(2k+2)\pi- \arccos x\right),
		\end{align*}
		which leads to \eqref{eq:cdfz1}.
		
		Clearly, if $N(t)=0$, $Z_1(t)=\cos (\pm ct)=\cos ct$ with probability $e^{-\lambda t},$ which leads to the singular component of the distribution of $Z_1(t),$ that is $\frac12e^{-\lambda t}\delta_{\cos ct}(x).$
		We observe that the absolutely continuous part of the cumulative distribution function \eqref{eq:cdfz1} is given by 
		$$\int_{-\infty}^x \nu(y,t)\mathds1_{|y|<1}\de y=\sum_{k=-\infty}^{\infty}\int_{\arccos x+2k\pi}^{(2k+2)\pi- \arccos x}\mu(y,t)\mathds1_{|y|<ct}\de y,$$
		from which we derive \eqref{eq:densz1}.

		2) For the process $\{Z_2(t),\,t\geq 0\},$ we can obtained \eqref{eq:cdfz2} as follows
		\begin{align*}
		\mathds P(Z_2(t)\leq x)&=\sum_{k=-\infty}^{\infty}\mathds P\left(\sin X(t)\leq x,(2k-1/2)\pi\leq X(t)\leq (2k+1/2)\pi\right) \\
		&\quad+\sum_{k=-\infty}^{\infty}\mathds P\left(\sin X(t)\leq x, (2k+1/2)\pi\leq X(t)\leq (2k+3/2)\pi\right) \\
		&=\sum_{k=-\infty}^{\infty}\mathds P\left( X(t)\leq \arcsin x +2k\pi,(2k-1/2)\pi\leq X(t)\leq (2k+1/2)\pi\right) \\
		&\quad+\sum_{k=-\infty}^{\infty}\mathds P\left( X(t)\geq (2k+1/2)\pi+\pi/2-\arcsin x, (2k+1/2)\pi\leq X(t)\leq (2k+3/2)\pi\right) \\
		&=\sum_{k=-\infty}^{\infty}\left[\mathds P\left((2k-1/2)\pi\leq X(t)\leq \arcsin x+2k\pi\right)\right.\\
		&\quad+\left.\mathds P\left((2k+1)\pi-\arcsin x\leq  X(t)\leq  (2k+3/2)\pi\right)\right],
		\end{align*}
		where $|x|<1.$ 
		
		We have that $\mathds P(Z_2(t)=\sin(\pm ct))=\frac12\mathds P(N(t)=0)=\frac12e^{-\lambda t},$ which leads to the singular component of the distribution of $Z_2(t),$ namely $\frac12e^{-\lambda t}[\delta_{\sin ct}(x)+\delta_{-\sin ct}(x)]$ We observe that the absolutely continuous part of the cumulative distribution function \eqref{eq:cdfz2} is equal to
		\begin{align*}
		\int_{-\infty}^x \eta(y,t)\mathds1_{|y|<1}\de y&=\sum_{k=-\infty}^{\infty}\left[\int_{(2k-1/2)\pi}^{\arcsin x+2k\pi}\mu(y,t)\de y+\int_{(2k+1)\pi-\arcsin x}^{ (2k+3/2)\pi}\mu(y,t)\de y\right],
		\end{align*}
		and then \eqref{eq:densz2} holds.
	\end{proof}

	Starting from the aforementioned weak convergence result of the telegraph process to the Brownian motion it is possible to derive an asymptotic diffusion approximation for the dynamics of the random harmonic oscillator here considered. Let \[  Z_{1,n}(t) =  \cos ( X_n(t)) ,\,\, Z_{2,n}(t) =  \sin ( X_n(t)) ,  \qquad n \in \mathbb N, \, t \in [0,1].\]
	
	\begin{proposition}
		We have that 
		\begin{enumerate}[(i)]
			\item (weak convergence) 
			\[ Z_{1,n} \Longrightarrow \mathcal V  ,\,\, Z_{2,n} \Longrightarrow \mathcal Y \]
			where $\mathcal V(t) = \cos B(t)$ and $\mathcal Y(t) = \sin B(t), t \in [0,1]$.
			\item (diffusion approximation) The limiting process $ \{ (\mathcal V(t), \mathcal Y(t)), \,t\geq 0 \} $  satisfies the following 
			system of SDEs
			\begin{equation}\label{eq:diff-approx}
			\begin{cases}
			\mathrm d \mathcal Y(t) &= -\frac 12 \mathcal Y(t) \mathrm d t + \mathcal V(t) \mathrm d B(t) \\
			\mathrm d \mathcal V(t) &= -\frac 12 \mathcal V(t) \mathrm d t - \mathcal Y(t) \mathrm d B(t)
			\end{cases}
			\end{equation}
			with $\mathcal Y(0)=0$, $\mathcal V(0)=1$.
		\end{enumerate}
	\end{proposition}
	\begin{proof}
		The weak convergence result is again an application of Theorem 1.2 \cite{hor} or Theorem 2.3 in \cite{ghosh} and of the continuous mapping theorem. Point\emph{(ii)} is a straightforward 	application of It\^o's Lemma.
	\end{proof}

	\section{Some generalizations of the wrapped telegraph process}

	\subsection{Asymmetric circular telegraph process}
	
	The asymmetric telegraph process has been considered in \cite{beghin}. 
	It is a telegraph process where the particle is allowed to move forward or backward with two different 
	velocities, $ c_1, \, c_2 $. Moreover the process is allowed to have two different velocity switching rates $ \lambda_1 $ and $ \lambda_2 $. Thus the underlying telegraph signal is modeled as a continuous time Markov chain $ \{ V(t),\,t\geq 0\} $ with state space $ \{c_1, -c_2\} $, where
	$ \lambda_i $ is the rate of the exponential waiting time when the telegraph signal is in state $ (-1)^{i+1}c_i, \, i=1,2$, and $ V(0) $ is uniformly distributed on $ \{c_1, -c_2\} $. The asymmetric telegraph process is then defined as $ X(t) = \int_0^t V(s) \de s $. Clearly the classical telegraph process is recovered as a particular case when
	$ c_1 = c_2 = c, \lambda_1 = \lambda_2= \lambda $. 
	
	The asymmetric process admits the representation
	
	\begin{equation}\label{eq:asym-tel-def}
	X(t)  = 
	\int_0^t \left\{ 
	\frac{c_1 - c_2}{2} + \frac{c_1 + c_2}{2} 
	\left[
	\mathds 1_{\{V(0)=c_1\}} - 
	\mathds 1_{\{V(0)=-c_2\}} 
	\right]
	(-1)^{N(s)}
	\right\}
	\de s
	\end{equation}
	where  $ \{N(t), t\geq 0\} $ is defined as the counting process $ N(t) = \min \{ n\in \mathbb N: \sum_{j=1}^n S_j > t\} $ and the random times $ \{S_n, n \in \mathbb N\} $ are conditionally independent given $ V(0) $, 
	with the following conditional law, under $ V(0) = (-1)^{i+1} c_i $:
	
	\begin{align*}
	&\{S_{2k-1}, k\in \mathbb N\} \text{ are exponentially distributed with mean } \frac{1}{\lambda_i} \\
	&\{S_{2k}, k\in \mathbb N\}\text{ are exponentially distributed with mean } \frac{1}{\lambda_{i + (-1)^{i+1}}}
	\end{align*}
	$ i = 1,2 $.
	
	The distribution of $ X(t) $ can be obtained by means of relativistic space-time tranformations (\cite{beghin}) and is given by 
	\begin{align}\label{eq:asy-tel-dist}
	p(\de x, t) = 
	\frac 12 e^{-\lambda_1t} \delta_{-c_2 t}(\de x)  + 
	\frac 12 e^{-\lambda_2t} \delta_{c_1 t}(\de x) + 
	\mu(x, t) \mathds 1_{-c_2 t < x < c_1t} \, \de x
	\end{align}
	where the density $ \mu $ of the absolutely continuous component of $ p $ reads
	\begin{align}\label{eq:asy-tel-ac}
	\mu(x,t) & = \frac{1}{c_1+c_2} 
	e^{- \frac 12 (\lambda_1 + \lambda_2) t +
		\frac{\lambda_2 - \lambda_1 }{c_2 + c_1}x +
		\frac{(\lambda_2 - \lambda_1)(c_2 - c_1)}{2(c_2 + c_1) } t    } 
	\\
	& \notag 
	\times \Bigg[
	\frac{\lambda_1 + \lambda_2 }{2} 
	I_0 \left(  2 \frac{\sqrt{\lambda_1 \lambda_2}}{c_2 + c_1}
	\sqrt{(x+c_2 t) (c_1 t - x)} \right) 
	+ 
	\frac{\partial}{\partial t} 
	I_0 \left(  2 \frac{\sqrt{\lambda_1 \lambda_2}}{c_2 + c_1}
	\sqrt{(x+c_2 t) (c_1 t - x)} \right)   \\
	& 
	\quad - \frac{(c_2 - c_1)}{2} 
	\frac{\partial}{\partial x} I_0 \left(  2 \frac{\sqrt{\lambda_1 \lambda_2}}{c_2 + c_1}
	\sqrt{(x+c_2 t) (c_1 t - x)} \right) 
	\Bigg]. \notag 
	\end{align}

	Let $Z(t) = ze^{iX(t)} \,, t \geq 0$
	where $ X $ now denotes the asymmetric process \eqref{eq:asym-tel-def}. The process $ \{Z(t), \, t\geq 0\} $ then defines the \textit{asymmetric} circular telegraph process with starting point $ Z(0) = z \in \mathbb D $.
	It describes the motion of a particle on the unit circle with two different angular velocities. The particle rotates counter-clockwise with constant angular speed $ c_1 $ and clockwise with angular speed of modulus $ c_2 $.
	Moreover in this model allows the particle to exhibit a preferential direction of rotation: it rotates 
	counter-clockwise and clockwise, on average, for a time period of $ 1/\lambda_1 $ and $ 1/\lambda_2  $ , respectively, before each velocity flip. The velocity process is described by the continuous time Markov chain $\{ V(t), t \geq 0\} $ defined above.
	The probability distribution of $ Z(t) $ can be easily obtained by wrapping the distribution \eqref{eq:asy-tel-dist}
	over the unit circle, again identified with the interval $ [0,2 \pi) $. In fact we have
	\begin{align}
	& p_Z(\de \theta, t)  \\ 
	&= 
	\sum_{k=-\infty }^\infty
	\frac12 [e^{-\lambda_1 t}\delta_{- c_1t - 2 k \pi}(\de \theta)+ e^{-\lambda_2 t} \delta_{-c_2t - 2 k \pi}(\de \theta)]+
	\sum_{k=-\infty }^\infty
	\mu(\theta + 2 k \pi,t) \mathds{1}_{-c_2 t < \theta + 2 k \pi <c_1t}\,\de \theta \notag 
	\end{align}

	By a slight modification of the argument of Theorem 2 it is immediate to check that 
	the infinitesimal generator  $\mathfrak L$ of $\{(Z(t),V(t)), t\geq 0\}$ in the \textit{asymmetric case} is given by 
	\begin{equation}\label{eq:asymm-gen}
	\left\{
	\begin{aligned}
	\mathfrak Lf(z,c_1) & =iz c_1 f'(z,v)+ \lambda_1f(z,-c_2)- \lambda_1 f(z,c_1) \\
	\mathfrak Lf(z,c_2) & =iz c_2 f'(z,v)+ \lambda_2f(z, c_1)- \lambda_2 f(z,-c_2)
	\end{aligned}
	\right.
	\end{equation}
	with Dom$(\mathfrak L)=\{f\in B(\mathbb G):z\mapsto f(z, v)\in A(\mathbb D)\}$, where in this case $ \mathbb G = \mathbb D \times \{c_1, -c_2\}  $.
	The operator in  \eqref{eq:asymm-gen} can be equivalently expressed in compact matrix form as
	\begin{equation}
	\mathfrak L  = 
	i z 
	\begin{pmatrix}
	c_1 \frac{\de }{\de z} & 0 \\
	0 & c_2 \frac{\de }{\de z}
	\end{pmatrix} + 
	\begin{pmatrix}
	-\lambda_1 & \lambda_1 \\
	\lambda_2 & -\lambda_2
	\end{pmatrix} 
	\end{equation}
	with the identification of $ \text{Dom}(\mathfrak L) $ with  $ A(\mathbb D) \times A(\mathbb D) $.
	
	The absolutely continuous component $ \mu $ of the probability law of the asymmetric telegraph process
	can satisfies the following hyperbolic equation (\cite{beghin})
	
	\begin{align}\label{eq:gov-asym}
	\frac{\partial^2 \mu }{\partial t^2} &= 
	c_1 c_2 \frac{\partial^2 \mu }{\partial x^2} 
	+ (c_2 -c_1) \frac{\partial^2 \mu }{\partial t \partial x}
	- (\lambda_1 + \lambda_2) 	\frac{\partial \mu }{\partial t}  \\
	& \quad 
	+ \frac 12 \left[
	(c_2 - c_1) (\lambda_1 + \lambda_2) - (\lambda_2 - \lambda_1)(c_1 + c_2)
	\right]
	\frac{\partial \mu }{\partial x}. \notag
	\end{align}
	
%
%

	It is possible to see that by taking the limit in \eqref{eq:gov-asym} 
		under Kac-type conditions, one obtains the governing equation of a Brownian motion with drift.
		In fact by taking the limits for $ \lambda_i, c_i \to \infty $ in such a way that 
	
	\begin{equation}\label{eq:asym-cond}
		\frac{\lambda_1}{\lambda_2} \to  \nu^2 >0 , \quad 
		\frac{c_i}{\sqrt{\lambda_i}} \to \sigma_i , \, i=1,2 , \quad
		\frac{\lambda_2 c_1 - \lambda_1 c_2}{\lambda_1 + \lambda_2 } \to \delta  
	\end{equation}
	it is possible to show that the marginal distributions of the asymmetric telegraph process converges to a 
	drifted Brownian motion
	\begin{equation}\label{eq:asym-conv}
		X(t) \xrightarrow{d} \sigma W(t) + \delta t , \quad t>0
	\end{equation}
	where $ W(t) $ denoted a standard Brownian motion and
	\[
	\sigma = \frac{\sigma_1 \sigma_2}{\sqrt{(\sigma_1^2 + \sigma_2^2)/2}} \, .
	\]
	
	By applying the continuous mapping theorem it follows immediately that the marginal distributions of the asymmetric circular telegraph process, under the conditions \eqref{eq:asym-cond}, converge to those of a circular Brownian motion with drift, i.e.
	\begin{equation}
		Z(t) \xrightarrow{d} ze^{i\delta t} [\mathfrak B(t)]^\sigma \,,\quad t > 0
	\end{equation}
	where $ \mathfrak B $ denoted a standard circular Brownian motion.
	
	\subsection{Circular motion with heay-tailed intertimes}
	Let us consider another type of telegraph random evolution on the real line. We define
	\begin{align}
	\mathfrak X(t):=V_0\sum_{k=1}^{M_t-1}(-1)^{k-1} D_k+(t-T_{M_t})V_0 (-1)^{M_t-1}
	\end{align}
	where $\{M_t,\,t\geq 0\}$ is a counting process defined as follows 
	$$M_t:=\inf\{k\geq0: T_k>t\}$$ We assume that $D_1$ is in the domain of attraction of a stable law with index $\alpha, 0<\alpha<2,\alpha\neq 1$ ($\mathds E(D_1^2)=\infty$);  i.e.
	$$\mathds P(D_1>x)=x^{-\alpha} h(x),\quad x\geq 1,$$  
	where $h(x)$ is a slowly varying function at $\infty.$ Therefore $\{\mathfrak X(t),t\geq 0\}$ represents a telegraph motion with  heay-tailed intertimes.

	 Let $\{S_\alpha^{(1)}(t),\,t\geq 0\}$ and $\{S_\alpha^{(2)}(t),\,t\geq 0\}$ be two independent, identically distributed completely asymmetric stable processes; i.e. 
	$${\mathds  E}(e^{iu S_\alpha^{(j)}(t)})=\exp\left(-t|u|^\alpha\left(1-i\, \sgn(u)\tan\left(\frac{\pi\alpha}{2}\right)\right)\right),\quad u\in\mathbb R,\quad  j=1,2.$$
	By setting $S_\alpha(t):=S_\alpha^{(1)}(t/2)-S_\alpha^{(2)}(t/2)$ and $U_\alpha(t):=S_\alpha^{(1)}(t/2)+S_\alpha^{(2)}(t/2).$ We observe that $\{S_\alpha(t),\,t\geq 0\}$ represents a symmetric stable process with
	$${\mathds  E}(e^{iu S_\alpha(t)})=\exp\left(-t|u|^\alpha\right)$$
	and $\{U_\alpha(t),\,t\geq 0\}$ is an asymmetric stable process; that is $U_\alpha(t)\stackrel{\text{law}}{=}S_\alpha^{(j)}(t).$
	Furthermore we define the inverse process $\{L_\alpha (t),t\geq 0\}$ of $U_\alpha(t)$ as follows
	$$L_\alpha(t):=\inf\{s: U_\alpha(s)>t\}.$$
	Now, we introduce the wrapped up telegraph process $\{W(t), t\geq 0\}$ on $\mathbb S$ (starting at $W(0)=1$) as follows
	\begin{equation}
	W(t):=e^{i \mathfrak X (t)}.
	\end{equation}
		
		The next theorem regards the asymptotic diffusion regime for $\{\mathfrak X(t),t\geq 0\}.$
	\begin{theorem}
		We have that
		\begin{equation}
		(W(nt))^{1/n}\Longrightarrow \mathfrak S_\alpha(L_\alpha(t)),\quad \text{if}\,\, 0<\alpha<1,
		\end{equation}
		and
		\begin{equation}
		(W(nt))^{1/a_n}\Longrightarrow\mathfrak S_\alpha(t/\mu),\quad \text{if}\,\, 1<\alpha<2,
		\end{equation}
		where $a_n$ satisfies
		$$\lim_{n\to\infty}\frac{n h(a_n)}{a_n^\alpha}=\frac{\Gamma(2-\alpha)|\cos(\pi\alpha/2)|}{\alpha-1}$$
		and $\mu=\mathds E(D_1),$ where $\{\mathfrak S_\alpha(t),t\geq 0\}$ represents the circular stable process   
	\begin{equation}
	\mathfrak S_\alpha(t):=e^{i S_\alpha(t)}.
	\end{equation}

	\end{theorem}
	\begin{proof}  The proof follows exploits Theorem 1.3 in \cite{hor} and the same arguments of the proof of point $(ii)$ in Theorem \ref{th:conv}.
	\end{proof}
	
	Let $\mathfrak f_\alpha(x,t)$ and $\mathfrak f_\alpha(x,u,t,s)$ be the density function of $L_\alpha(t),t\geq 0,$ and $(L_\alpha(t), L_\alpha(s)), t\geq s\geq 0, $ respectively. It is interesting to derive the moments of limit processes appearing in the previous theorem. 
	\begin{theorem}
		For $0<\alpha<1,$
		\begin{equation}\label{eq:circstab1}
		{\mathds  E}\left[\mathfrak S_\alpha(L_\alpha(t))\right]^q=E_{\alpha,1}(-q^\alpha t^\alpha),\quad q>0,
		\end{equation}
		where $E_{\alpha,1}(x)=\sum_{k=0}^\infty \frac{x^k}{\Gamma(\alpha k+1)}, x\in\mathbb R,$ is the Mittag-Leffler function and
		\begin{align}\label{eq:circstab2}
		&{\mathds  E}[\mathfrak S_\alpha(L_\alpha(t))^{q_1}\mathfrak S_\alpha(L_\alpha(s))^{q_2}]\\
		&={\mathds  E}\exp\left\{-\varpi^\alpha\max\{L_\alpha(t),L_\alpha(s)\}- [(q_1+q_2)^\alpha-\varpi^\alpha]\min\{L_\alpha(t),L_\alpha(s)\}\right\} ,\notag
		\end{align}		where $t>0,s>0,$ $q_1>0,q_2>0,$ and $\varpi:=q_1 \mathds 1_{L_\alpha(t)>L_\alpha(s)}+q_2\mathds 1_{L_\alpha(t)\leq L_\alpha(s)}.$

		For $1<\alpha<2,$
		\begin{equation}\label{eq:circstab3}
		{\mathds  E}[\mathfrak S_\alpha(t)]^q=e^{-	tq^\alpha},\quad q>0,
		\end{equation}
		and 
		\begin{equation}\label{eq:circstab4}
		{\mathds  E}
		[\mathfrak S_\alpha(t)^{q_1}\mathfrak S_\alpha(s)^{q_2}]=\exp\left\{-\eta^\alpha\max\{t,s\}- [(q_1+q_2)^\alpha-\eta^\alpha]\min\{ t,s\}\right\},		\end{equation}
		where $\eta:=q_1 \mathds 1_{t>s}+q_2\mathds 1_{t\leq s}.$
	\end{theorem}
	\begin{proof} For the result \eqref{eq:circstab1}, we can write down
		\begin{align*}
		{\mathds  E}[\mathfrak S_\alpha(L_\alpha(t))]^q&=\int_0^\infty {\mathds  E}[e^{iq S_\alpha(u)}]\mathfrak f_\alpha(u,t)\de u\\
		&=\int_0^\infty e^{-u q^\alpha}\mathfrak f_\alpha(u,t)\de u\\
		&=E_{\alpha,1}(-q^\alpha t^\alpha),
		\end{align*}
		where in the last step we have used the result on Laplace transform for the inverse process contained in \cite{bingham}.  Moreover, we get
		\begin{align*}
		&{\mathds  E}[\mathfrak S_\alpha(L_\alpha(t))^{q_1}\mathfrak S_\alpha(L_\alpha(s))^{q_2}]\\
		&={\mathds  E}[\mathfrak S_\alpha(L_\alpha(t))^{q_1}\mathfrak S_\alpha(L_\alpha(s))^{q_2}\mathds 1_{L_\alpha(t)>L_\alpha(s)}]+{\mathds  E}[\mathfrak S_\alpha(L_\alpha(t))^{q_1}\mathfrak S_\alpha(L_\alpha(s))^{q_2}\mathds 1_{L_\alpha(t)\leq L_\alpha(s)}]&\\
		&=\iint_{u>v\geq 0} {\mathds  E}[\mathfrak S_\alpha(u)^{q_1}\mathfrak S_\alpha(v)^{q_2}]\mathfrak f_\alpha(u, v, t,s)\de u\de v+\iint_{v\geq u\geq 0} {\mathds  E}[\mathfrak S_\alpha(u)^{q_1}\mathfrak S_\alpha(v)^{q_2}]\mathfrak f_\alpha(v, u,s,t)\de u\de v\\
		&=\iint_{u>v\geq 0}  {\mathds  E}[e^{iq_1 S_\alpha(u)+ iq_2S_\alpha(v)}]\mathfrak  f_\alpha(u, v, t,s)\de u\de v+\iint_{v\geq u\geq 0}  {\mathds  E}[e^{iq_1 S_\alpha(u)+ iq_2S_\alpha(v)}]\mathfrak \mathfrak f_\alpha(v, u,s,t)\de u\de v\\
		&=\iint_{u>v\geq 0}  {\mathds  E}(e^{iq_1 (S_\alpha(u)-S_\alpha(v))+ i(q_2+q_1)S_\alpha(v)})\mathfrak  f_\alpha(u, v, t,s)\de u\de v\\
		&\quad+\iint_{v\geq u\geq 0}   {\mathds  E}(e^{iq_2 (S_\alpha(v)-S_\alpha(u))+ i(q_2+q_1)S_\alpha(u)})\mathfrak \mathfrak f_\alpha(v, u,s,t)\de u\de v\\
		&=\iint_{u>v\geq 0}  e^{-q_1^\alpha u- [(q_2+q_1)^\alpha -q_1^\alpha] v}\mathfrak  f_\alpha(u, v, t,s)\de u\de v+\iint_{v\geq u\geq 0}   e^{-q_2^\alpha v- [(q_2+q_1)^\alpha -q_2^\alpha] u}\mathfrak \mathfrak f_\alpha(v, u,s,t)\de u\de v,
		\end{align*}
		 which leads to \eqref{eq:circstab2}. The proof of the \eqref{eq:circstab3} and \eqref{eq:circstab4} is now straightforward.
	\end{proof}
	
	\begin{remark}
		In \cite{toaldo}, the authors studied $\{\mathfrak S_\alpha(t),\,t\geq 0\}$ and $\{\mathfrak S_\alpha(L_\alpha(t)),\,t\geq 0\}.$ In particular, they proved that for any  $t>0$
		$$\mathfrak S_\alpha(t)\stackrel{\text{law}}{=}\mathfrak B(2 S_{\alpha/2}^{(j)}(t))$$
		with density function given by
		$$\mu_{\mathfrak S_\alpha}(\theta,t)=\frac{1}{2\pi}+\frac1\pi\sum_{k=1}^\infty e^{-k^\alpha t}\cos(k\theta).$$
		
		Let us introduce the one-dimensional fractional Laplace operator 
		\begin{equation}
		-\left(-\frac{\partial ^2}{\partial x^2}\right)^{\alpha/2}=\frac{\sin(\pi\alpha/2)}{\pi}\int_0^\infty \left(\lambda +\left(-\frac{\partial ^2}{\partial x^2}\right)\right)^{-1}\lambda^{\alpha/2}\de \lambda.
		\end{equation}
		The density function  $\mu_{\mathfrak S_\alpha}$ solves the space-fractional equation
		$$\frac{\partial u}{\partial t}(\theta,t)=-\left(-\frac{\partial ^2}{\partial \theta^2}\right)^{\alpha/2}u(\theta,t),\quad \theta\in[0,2\pi),t\geq 0,$$
		subjected to the initial condition $u(\theta,0)=\delta(\theta).$
	\end{remark}

\bibliographystyle{abbrv}

\end{document}